\documentclass[12pt]{article}
\usepackage[utf8]{inputenc}
\usepackage[T1]{fontenc}
\usepackage{amsmath, amsthm, amssymb}
\usepackage{physics} 
\usepackage{braket} 
\usepackage{esint}
\usepackage{enumitem}
\usepackage[a4paper, margin=48pt]{geometry}
\usepackage[x11names]{xcolor}
\usepackage{hyperref} 
\usepackage{graphicx} 
\newcommand{\mres}{\mathbin{\vrule height 1.6ex depth 0pt width 0.13ex\vrule height 0.13ex depth 0pt width 0.8ex}}

\theoremstyle{plain}
\newtheorem{thm}{Theorem}[section]
\newtheorem{prop}[thm]{Proposition}
\newtheorem{cor}[thm]{Corollary}
\newtheorem{lem}[thm]{Lemma}
\newtheorem*{thm*}{Theorem}
\newtheorem*{prop*}{Proposition}
\newtheorem*{cor*}{Corollary}
\newtheorem*{lem*}{Lemma}

\theoremstyle{definition}
\newtheorem{defi}[thm]{Definition}

\newtheorem*{defi*}{Definition}
\newtheorem*{exa*}{Example}

\theoremstyle{remark}
\newtheorem{rmk}[thm]{Remark}
\newtheorem*{rmk*}{Remark}

\newcommand{\R}{\mathbf{R}}

\newcommand{\HH}{\mathcal{H}}

\newcommand{\C}{\mathbb{C}}
\newcommand{\CC}[1]{\C #1:#1}

\usepackage[toc,page]{appendix} 

\title{Optimal regularity for quasiminimal sets of codimension one in $\R^2$ and $\R^3$}

\author{C. Labourie, Y. Teplitskaya}

\date{}

\begin{document}

\maketitle

\abstract{
    Quasiminimal sets are sets for which a pertubation can decrease the area but only in a controlled manner. We prove that in dimensions $2$ and $3$, such sets separate a locally finite family of local John domains. Reciprocally, we show that this property is a sufficient for quasiminimality. In addition, we show that quasiminimal sets locally separate the space in two components, except at isolated points in $\R^2$ or out a of subset of dimension strictly less than $N-1$ in $\R^N$.
}



\tableofcontents

\section{Introduction}

Minimal sets are a central focus in classical geometric measure theory and in variational problems involving a surface term. Inspired by soap films, their area is minimal compared to admissible pertubations. In contrast, the area of quasiminimal sets can decrease but only to a limited extent. This notion allows to represent sets minimizing inhomogeneous and possibly highly irregular energies. Building on Almgren’s foundational work \cite{Alm76}, subsequent research by David and Semmes established uniform bounds for their geometric structure, such as uniform rectifiability \cite{DS20}. This paper is motivated in particular by \cite{DS98}, where David and Semmes characterize domains with a quasiminal boundary as bi-John domains with an Ahlfors-regular boundary. We prove an analogue optimal regularity theorem for quasiminimal sets in $\R^2$ and $\R^3$.

The primary difference between the setting of \cite{DS98} and ours is that the boundary of a connected domain $W$ separates the space in two regions: $W$ and $\R^N \setminus W$, whereas a quasiminimal set may separate the space into multiple, potentially infinitely many, components.
In \cite{DP22}, David and Pourmohammad extended the techniques of \cite{DS98} to the case of a \emph{finite} Caccioppoli partition minimizing a general energy. They proved that such a partition is composed of local John domains but with constants depending on the number of partition elements. Therefore, a central novelty and challenge in our work is to show that quasiminimal sets locally separates a finite number of components, with a controlled upper bound on the number of these components.

The local finiteness of minimal partitions was previously studied by Tamanini, Massari, Congedo and Leonardi in a serie of works \cite{Tamanini91, Tamanini93, Tamanini96, Tamanini98}. Their techniques rely however on arguments specific to minimal sets, such as the fact that their blow-up limits are cones. Unlike minimal sets, quasiminimal sets lack Euler-Lagrange equations, monotonicity formulas, $\varepsilon$-regularity theorems and their blow-up limits are not cones in general. This necessitates new techniques which are robust enough to apply to a broader setting.

In Section \ref{section_standard}, we provide simple proofs for standard regularity properties of quasiminimal sets such as Ahlfors-regularity and uniform rectifiability. In Section \ref{section_generic}, we show that almost-every point of a quasiminimal set is an ``interface point'' where two components meet. We establish our main results in Section \ref{section_finiteness} and Section \ref{section_john}, where we show that quasiminimal sets partition a domain into a locally finite family of local John domains. Reciprocally, we justify in Section \ref{section_sufficient} that this is a sufficient condition for quasiminimality. We finally investigate ``junction points'' where multiple components meet in Section \ref{section_dimension} and show that they are isolated in the plane and have a dimension $< N-1$ in $\R^N$. These results not only extend the theoretical understanding of quasiminimal sets but also offer new perspectives for applying these sets in complex geometric and variational contexts, such as image segmentation and fracture mechanics.

Our local finiteness theorem applies to the dimensions $N = 2$ and $3$ and it remains an open question whether quasiminimal sets locally separate a finite number of components in higher dimension.

\section{Definitions}

For the whole paper, we fix an open subset $\Omega$ of $\R^N$, where $N \geq 2$. Given a set $A$, the notation $A \subset \subset \Omega$ stands for $\overline{A} \subset \Omega$.
A \emph{coral set} $K \subset \Omega$ is a relatively closed subset of $\Omega$ such that for all $x \in K$ and for all $r > 0$, we have $$\HH^{N-1}(K \cap B(x,r)) > 0.$$
We now turn our attention to the definition of admissible competitors. There exists different notions of competitors in the literature but in general a competitor $F$ of a set $K$ in a ball $B$ should satisfy $F \setminus B = K \setminus B$ and $F$ should «span $K \cap \partial B$» in the same way as $K$ does. A typical class of competitors, introduced by Almgren \cite{Alm76}, are the images $f(K)$ of $K$ under a Lipschitz deformation $f : K \to \R^N$ such that $f = \mathrm{id}$ in $K \setminus B$ and $f(K \cap B) \subset B$. In this paper, we work with a class of competitor which is more convenient to deal with, called \emph{topological competitors}. This notion was introduced by Bonnet \cite{Bonnet} in the context of image segmentation.

A \emph{topological competitor} of $K$ in a ball $B(x_0,r) \subset \subset \Omega$ is a relatively closed subset $F \subset \Omega$ such that $F \setminus B(x_0,r) = K \setminus B(x_0,r)$ and
\begin{equation}\label{eq_topological}
    \begin{gathered}
        \text{for all points $x,y \in \Omega \setminus \left(B(x_0,r) \cup K\right)$,}\\
        \text{if $x$,$y$ are separated by $K$, then they are separated by $F$.}
    \end{gathered}
\end{equation}
This means that if $x$, $y$ lie in distinct connected components of $\Omega \setminus K$, they also lie in distinct connected components of $\Omega \setminus F$.
\begin{figure}[ht]
\begin{center}
\includegraphics[width=0.22\linewidth]{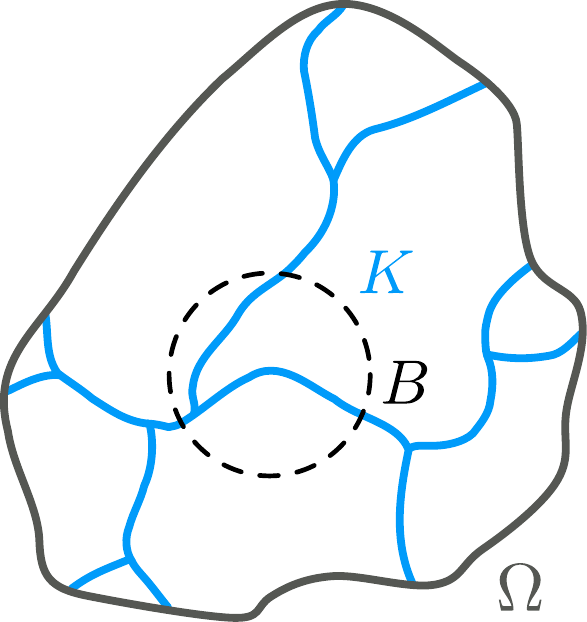}
\quad \quad
\includegraphics[width=0.22\linewidth]{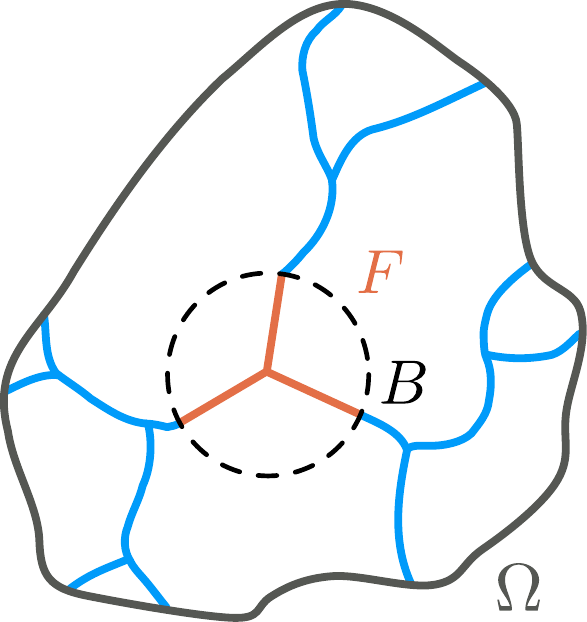}
\end{center}
\caption{A topological competitor of $K$ in a ball $B$.}
\end{figure}

\begin{defi}
    Let $M \geq 1$.
    A $M$-quasiminimal set is a coral set $K \subset \Omega$ which is $\HH^{N-1}$-locally finite in $\Omega$ and such that for all $x_0 \in K$, $r > 0$ such that $B(x_0,r) \subset \subset \Omega$ and for all topological competitor $F$ of $K$ in $B(x_0,r)$, we have
    \begin{equation}\label{eq_quasi}
        \HH^{N-1}(K \setminus F) \leq M \HH^{N-1}(F \setminus K).
    \end{equation}
\end{defi}
In the case $M = 1$, such a set is called a \emph{minimal set} (there are also different notions of minimal sets in the literature, for instance \cite{Alm76, Reifenberg, DS20, HaPu}).
\begin{rmk}
Notice that (\ref{eq_quasi}) implies a weaker quasiminimality property, namely
\begin{equation}\label{eq_quasi2}
    \HH^{N-1}(K \cap B) \leq M \HH^{N-1}(F \cap B).
\end{equation}
Property (\ref{eq_quasi2}) will be sufficient for some of our first results such as Ahlfors-regularity but not for the rest of the paper.
\end{rmk}

\begin{rmk}[Motivation]
    Quasiminimal sets represent sets minimizing Borel regular measures $\mu$ in $\Omega$ such that
    \begin{equation}\label{eq_mu}
        \lambda \HH^{N-1} \leq \mu \leq \Lambda \HH^{N-1}
    \end{equation}
    for some constants $0 < \lambda \leq \Lambda$. Let us justify this claim.
    Let $K \subset \Omega$ be a local minimizer of $\mu$, that is, a coral $\HH^{N-1}$-locally finite set in $\Omega$ such that for all $x_0 \in K$, $r > 0$ with $B(x_0,r) \subset \subset \Omega$ and for all topological competitor $F$ of $K$ in $B(x_0,r)$, we have
    \begin{equation*}
        \mu(K \cap B(x_0,r)) \leq \mu(F \cap B(x_0,r)).
    \end{equation*}
    It follows that
    \begin{equation*}
        \mu(K \setminus F) \leq \mu(F \setminus K)
    \end{equation*}
    and then by (\ref{eq_mu}) that
    \begin{equation}\label{eq_J}
        \HH^{N-1}(K \setminus F) \leq M \HH^{N-1}(F \setminus K), \quad \text{where $M = \Lambda/\lambda \geq 1$.}
    \end{equation}
    Reciprocally (\ref{eq_J}) directly translates as
    \begin{equation*}
        \mu(K \cap B(x_0,r)) \leq \mu(F \cap B(x_0,r)),
    \end{equation*}
    where $\mu = \HH^{N-1} \mres K + M \HH^{N-1} \mres (\R^N \setminus K)$.
\end{rmk}

\begin{rmk}[Quasiminimal sets in fracture mechanics]
    Another motivation for the study of quasiminimal sets is their application to the theory of brittle fractures in linear elasticity.
    Let $\Omega$ be a bounded open set of $\R^N$ representing a homogeneous isotropic brittle solid.
    If a loading is applied on $\partial \Omega$, the solid will deform and absorb energy. But if this exceeds the material's limit, the solid will releave (totally or partially) this energy by making a crack.
    Based on the pioneering work of Griffith in the 1920's, Francfort and Marigo \cite{FM} formulated the equilibrium state of the fracture as the minimization of the Griffith functional
    \begin{equation*}
        \int_{\Omega \setminus K} \CC{e(u)} \dd{x} + \beta \HH^{N-1}(K),
    \end{equation*}
    over pairs $(u,K)$, where $K$ is a relatively closed subset of $\Omega$ of dimension $N-1$ and $u : \Omega \setminus K \to \R^N$ is a smooth vector field satisfying a Dirichlet condition along $\partial \Omega$. Here, $K$ is the fracture, $u$ is the displacement field, $e(u) := (\nabla u + \nabla u^T)/2$ is the linear strain tensor, $\C$ is an elasticity tensor and $\beta > 0$ is the fracture toughness (it quantifies the ability of the material to resist a crack propagation).
   The Griffith functional is a vectorial analogue of the Mumford-Shah functional \cite{AFP, DavidBOOK, DeLellisBOOK} which is more physically relevant in dimension $N = 3$.

    A particular case of interest is when $\int_{\Omega \setminus K} \CC{e(u)} \dd{x} = 0$; this corresponds to an asymptotic behavior of materials with no fracture toughness. In this case, it is standard that $K$ is a minimal set in $\Omega$. In the general case, fractures look like minimal sets at points $x \in K$ where the elastic energy $\int_{B(x,r)} \abs{e(u)}^2 \dd{x}$ becomes negligible compared to the surface term $\HH^{N-1}(K \cap B(x,r))$ when $r$ goes to $0$. This is actually the behavior of fractures at generic points since this holds almost-everywhere along the crack. This connection plays an important role in the regularity theory of Mumford-Shah minimizers, for instance \cite{DavidBOOK, DeLellisBOOK, AFH, DLF, DLFR, DPF}, and in their recent adaptations to the Griffith functional \cite{FLS, LL3}.

    For general brittle solid with inhomogeneous and possibly irregular fracture toughness, one would minimize
    \begin{equation*}
        \int_{\Omega \setminus K} \CC{e(u)} \dd{x} + \mu(K),
    \end{equation*}
    where $\mu$ is a Borel regular measure in $\Omega$ satisfying (\ref{eq_mu}). In that case, when there is no elastic energy or when it is neglible compared to the surface term, the fracture behaves like a quasiminimal set.
\end{rmk}

\begin{rmk}[Example]
    If $K$ is a quasiminimal set in $\R^N$ and $f : \R^N \to \R^N$ is a bi-Lipschitz map, then $f(K)$ is a also quasiminimal set in $\R^N$ (with a bigger constant). We deduce that Lipschitz graphs are quasiminimal sets because they can be written as bi-Lipschitz images of hyperplanes.

    This shows also that a blow-up limit of a quasiminimal set may not be a cone, neither unique in general. Indeed, let $K$ be the graph of $x \in \R \mapsto \mathrm{dist}(|x|,C)$, where $C = \set{2^i | i \in \mathbf{Z}}$. As a Lipschitz graph, $K$ is a quasiminimal set but one can see that $2K = K$ so for all $\lambda > 0$, the set $\lambda K$ concides with $\lim_{i \to +\infty} r_i^{-1} K$, where $r_i = \lambda^{-1} 2^{-i}$. Hence, $K$ has an infinite number of blow-up limits at $0$ which are not cones (see Fig.~\ref{ncn}).
\begin{figure}[ht]
\centerline{\includegraphics[width=0.8\linewidth]{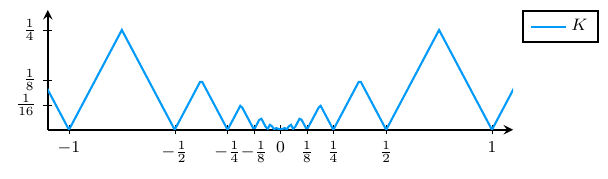}}
\caption{An example of quasiminimal set with different blow-up limits at the origin.}
\label{ncn}
\end{figure}
\end{rmk}

\section{Standard results}\label{section_standard}

Ahlfors-regularity and uniform rectifiability are well-known for Almgren quasiminimizers \cite{Alm76, DS20} of any codimension. These properties actually have a simpler proof when one consider quasiminimal of codimension 1 sets with respect to topological competitors. We present these proofs below for the sake of completeness and as some intermediate results will have an independent interest in the rest of the paper.
Let us make a few preliminary remarks.

\begin{rmk}
    Let $K$ be a quasiminimal set in $\Omega$.
    For all connected component $V$ of $\Omega \setminus K$, we have $\Omega \cap \partial V \subset K$, and as $K$ is $\HH^{N-1}$-locally finite, we deduce that $V$ has locally finite perimeter with $\Omega \cap \partial^* V \subset K$ (see \cite[Proposition 3.62]{AFP}).
\end{rmk}

\begin{rmk}
    If a relatively closed subset $F \subset \Omega$ satisfies the definition of topological competitors with $B(x_0,r)$ replaced by $\overline{B}(x_0,r)$, we say that $F$ is a topological competitor in $\overline{B}(x_0,r)$.
    This implies that $F$ is a topological competitor in all balls $B(x_0,t)$ such that $t > r$ and $\overline{B}(x_0,t) \subset \Omega$.
    Applying property (\ref{eq_quasi}) in such a ball $B(x_0,t)$, we still get
    \begin{equation*}
        \HH^{N-1}(K \setminus F) \leq M \HH^{N-1}(F \setminus K).
    \end{equation*}
    We can similarly apply (\ref{eq_quasi2}) in $B(x_0,t)$ such that $t > r$ and let $t \to r$ (using the fact that $K$ has locally finite measure in $\Omega$) to get
    \begin{equation*}
        \HH^{N-1}(K \cap \overline{B}(x_0,r)) \leq M \HH^{N-1}(F \cap \overline{B}(x_0,r)).
    \end{equation*}
\end{rmk}

\subsection{Ahlfors-regularity}\label{section_AR}

We start by observing that any ball centred on $K$ must separate at least two components of $\Omega \setminus K$.
\begin{lem}\label{lem_separation}
    Let $K$ be a $M$-quasiminimal set in $\Omega$. For all $x_0 \in K$ and $r > 0$ such that $B(x_0,r) \subset \Omega$,
    \begin{equation*}
        \text{$B(x_0,r)$ meets at least two components of $\Omega \setminus K$.}
    \end{equation*}
\end{lem}
\begin{proof}
    We proceed by contradiction and assume that there exists a connected component $V$ of $\Omega \setminus K$ such that
    \begin{equation*}
        B(x_0,r) \subset K \cup V.
    \end{equation*}
    Then, let us check that $F := K \setminus B(x_0,r/2)$ is a topological competitor of $K$ in $B(x_0,r/2)$.
    We consider two points $x$ and $y$ in $\Omega \setminus (B(x_0,r/2) \cup K)$ which are connected by a path in $\Omega \setminus F$ and we want to show that $x$ and $y$ are connected in $\Omega \setminus K$.
    There are two possibilities. If the path does not pass through $\partial B(x_0,r/2)$, it stays in $\Omega \setminus \overline{B}(x_0,r/2)$ where $F$ and $K$ coincide so $x$ and $y$ are connected in $\Omega \setminus K$.
    If the path passes through $\partial B(x_0,r/2)$ this can only be at a point of $\partial B(x_0,r/2) \setminus K \subset V$. Considering the portion of the part that starts from $x$ and meets $\partial B(x_0,t)$ for the first time, we see that $x$ is connected to $V$ in $\Omega \setminus K$ and thus $x \in V$. Considering similarly the portion of the path that leaves $\partial B(x_0,t)$ for the last time and arrive at $y$, we see that $y$ is connected to $V$ in $\Omega \setminus K$ so $y \in V$. We conclude that $x$ and $y$ are connected in $\Omega \setminus K$ and that $F$ is a topological competitor of $K$ in $B(x_0,r/2)$.

    We finally apply the quasiminimality property (\ref{eq_quasi2}) in $B(x_0,r/2)$, which gives
    \begin{equation*}
        \HH^{N-1}(K \cap B(x_0,r/2)) \leq M \HH^{N-1}(F \cap B(x_0,r/2)),
    \end{equation*}
    and thus
    \begin{equation*}
        \HH^{N-1}(K \cap B(x_0,r/2)) = 0.
    \end{equation*}
    This contradicts the fact that $x_0 \in K$ and that $K$ is coral.
\end{proof}

In the following Lemma, we show that in every ball $B$ centred on $K$, a connected component of $B \setminus K$ cannot be too big.
\begin{lem}\label{lem_af_volume}
    Let $K$ be a $M$-quasiminimal set in $\Omega$. Then for all $x_0 \in K$, for all $r > 0$ such that $B(x_0,r) \subset \Omega$ and for all connected component $V$ of $\Omega \setminus K$, we have
    \begin{equation}\label{eq_af_volume}
        \abs{B(x_0,r) \setminus V} \geq C^{-1} r^N, 
    \end{equation}
    where $C \geq 1$ is a constant which depends only on $N$ and $M$.
\end{lem}

\begin{rmk}\label{rmk_isoperimetry}
    As an application of (\ref{eq_af_volume}) which will be very helpfup in the rest of the paper, notice that the relative isoperimetry inequality
    \begin{equation*}
        \min \left(\abs{B(x_0,r) \cap V}, \abs{B(x_0,r) \setminus V}\right) \leq C \HH^{N-1}(B(x_0,r) \cap \partial^* V)^{N/(N-1)}
    \end{equation*}
    simplifies to
    \begin{equation*}
        \abs{B(x_0,r) \cap V} \leq C \HH^{N-1}(B(x_0,r) \cap \partial^* V)^{N/(N-1)}.
    \end{equation*}
\end{rmk}

\begin{proof}
    Let $x_0 \in K$ and $r > 0$ such that $B(x_0,r) \subset \Omega$. Let $V$ be a component of $\Omega \setminus K$.
    We assume that
    \begin{equation}\label{eq_initV}
        \abs{B(x_0,r) \setminus V} \leq \varepsilon r^N,
    \end{equation}
    for a small $\varepsilon > 0$ which will be fixed later.
    According to the co-area formula,
    \begin{equation*}
        \abs{B(x_0,r) \setminus V} = \int_0^r \HH^{N-1}(\partial B(x_0,t) \setminus V) \dd{t}
    \end{equation*}
    so we can find a radius $t \in (r/2,r)$ such that
    \begin{equation}\label{eq_tV}
        \HH^{N-1}(\partial B(x_0,t) \setminus V) \leq C r^{-1} \abs{B(x_0,r) \setminus V} \leq C \varepsilon r^{N-1}.
    \end{equation}
    We then justify that
    \begin{equation*}
        F := \left(K \setminus B(x_0,t)\right) \cup \big(\partial B(x_0,t) \setminus V\big).
    \end{equation*}
    is a topological competitor of $K$ in $\overline{B}(x_0,t)$.
    We consider two points $x$ and $y$ in $\Omega \setminus (\overline{B}(x_0,t) \cup K)$ which are connected by a path in $\Omega \setminus F$ and we want to show that $x$ and $y$ are connected in $\Omega \setminus K$.
    If the path does not intersect $\partial B(x_0,t)$, then it stays in the complement of $\overline{B(x_0,t)}$, where $F$ coincides with $K$. In this case, the path does not meet $K$ and the points $x$ and $y$ are also connected in $\Omega \setminus K$.
    If the path path meets $\partial B(x_0,t)$, it can only be at a point of $V$. Considering the portion of the path starting from $x$ until the first time it meets $\partial B(x_0,t)$, we see that $x$ is connected to $V$ in the complement of $K$.
    Similarly, $y$ is connected to $V$ in the complement of $K$. As $V$ is connect and disjoint from $K$, the points $x$ and $y$ are also connected in $\Omega \setminus K$. We conclude that $F$ is a topological competitor of $K$ in $\overline{B}(x_0,t)$.

    We apply the quasiminimality property (\ref{eq_quasi2}) in $\overline{B}(x_0,t)$, which gives
    \begin{equation*}
        \HH^{N-1}(K \cap \overline{B}(x_0,t)) \leq M \HH^{N-1}(F \cap \overline{B}(x_0,t)),
    \end{equation*}
    whence by (\ref{eq_tV}),
    \begin{equation*}\label{eq_minV}
        \HH^{N-1}(B(x_0,r/2) \cap \partial^* V) \leq M \HH^{N-1}(\partial B(x_0,t) \setminus V) \leq C \varepsilon r^{N-1}\\
    \end{equation*}
    By assumption (\ref{eq_initV}), we can choose $\varepsilon$ small enough so that
    \begin{equation*}
        \abs{B(x_0,r/2) \setminus V} \leq \tfrac{1}{2} \abs{B(x_0,r/2)},
    \end{equation*}
    and thus we can apply the relative isoperimetric inequality to $B(x_0,r/2) \setminus V$ in the ball $B(x_0,r/2)$.
    This gives
    \begin{equation}\label{eq_isoV}
        \abs{B(x_0,r/2) \setminus V} \leq C \HH^{N-1}(B(x_0,r/2) \cap \partial^* V)^{N/(N-1)}.
    \end{equation}
    As $\partial^* V \subset \partial V \subset K$, we can combine (\ref{eq_minV}) and (\ref{eq_isoV}) to estimate
    \begin{equation*}
        \abs{B(x_0,r/2) \setminus V} \leq C \varepsilon^{N/(N-1)} r^N.
    \end{equation*}
    We finally choose $\varepsilon$ small enough so that $\abs{B(x_0,r/2) \setminus V} \leq \varepsilon r^N$.
    We can thus iterate this estimate and deduce that for all integer $k \geq 0$,
    \begin{equation*}
        \abs{B(x_0,2^{-k}r) \setminus V} \leq \varepsilon (2^{-k} r)^N.
    \end{equation*}
    A simple interpolation arguments allow to conclude that
    \begin{equation}\label{eq_conclusion_V}
        \abs{B(x_0,r) \setminus V} \leq \varepsilon r^N \implies \lim_{t \to 0} t^{-N} \abs{B(x_0,t) \setminus V} \leq 2^{N} \varepsilon.
    \end{equation}

    Let us assume now that we have $\abs{B(x_0,r) \setminus V} \leq \varepsilon (r/2)^N$.
    For all $x \in K \cap B(x_0,r/2)$, we have
    \begin{equation*}
        \abs{B(x,r/2) \setminus V} \leq \varepsilon (r/2)^N,
    \end{equation*}
    which implies by (\ref{eq_conclusion_V}),
    \begin{equation}\label{eq_conclusion_xV}
        \lim_{t \to 0} t^{-N} \abs{B(x,t) \setminus V} \leq 2^N \varepsilon.
    \end{equation}
    We take $\varepsilon$ a bit smaller so that the small volume of $B(x_0,r) \setminus V$ implies that $$\abs{V \cap B(x_0,r/2)} > 0.$$ According to Lemma \ref{lem_separation}, $B(x_0,r/2) \setminus K$ must meet another component of $\Omega \setminus K$ so we also have $\abs{B(x_0,r/2) \setminus V} > 0$.
    It follows that $\HH^{N-1}(B(x_0,r/2) \cap \partial^* V) > 0$ and, by the properties of the reduced boundary, we can find a point $x \in B(x_0,r/2) \cap \partial^* V \subset B(x_0,r/2) \cap K$ such that
    \begin{equation*}
        \lim_{t \to 0} \frac{\abs{B(x,t) \setminus V}}{\abs{B(x,t)}} = \frac{1}{2},
    \end{equation*}
    which contradicts (\ref{eq_conclusion_xV}) if $\varepsilon$ is chosen small enough once again.
\end{proof}

We finally arrive at the Ahlfors-regularity property for quasiminimal sets.
\begin{prop}[Ahlfors-regularity]\label{prop_af}
    Let $K$ be a $M$-quasiminimal set in $\Omega$. Then for all $x_0 \in K$ and for all $r > 0$ such that $B(x_0,r) \subset \Omega$, we have
    \begin{equation*}
        C^{-1} r^{N-1} \leq \HH^{N-1}(K \cap B(x_0,r)) \leq C r^{N-1},
    \end{equation*}
    where $C \geq 1$ is a constant that depends only on $N$ and $M$.
\end{prop}
\begin{proof}
    Let $x_0 \in K$ and $r > 0$ be such that $B(x_0,r) \subset \Omega$. We start by proving the upper density bound.
    For $t < r$, we consider the relatively closed subset $F \subset \Omega$ defined by
    \begin{equation*}
        F := \left(K \setminus B(x_0,t)\right) \cup \partial B(x_0,t).
    \end{equation*}
    Then, $F$ is a topological competitor of $K$ in $\overline{B}(x_0,t)$ and we deduce by quasiminimality (\ref{eq_quasi2}),
    \begin{equation*}
        \HH^{N-1}(K \cap \overline{B}(x_0,t)) \leq M \HH^{N-1}(\partial B(x_0,t)) \leq C t^{N-1},
    \end{equation*}
    where $C \geq 1$ depends on $N$ and $M$.
    Letting $t \to r$, we deduce the upper bound
    \begin{equation*}
        \HH^{N-1}(K \cap B(x_0,r)) \leq C r^{N-1}.
    \end{equation*}

    We now pass to the lower bound.
    We let $(V_i)_{i \geq 0}$ denote the connected components of $\Omega \setminus K$, ordered in such a way that $$\abs{V_i \cap B(x_0,r)} \geq \abs{V_{i+1} \cap B(x_0,r)}.$$
    We show that
    \begin{equation}\label{eq_isoperimetric}
        \abs{B(x_0,r) \setminus V_0} \leq C \HH^{N-1}(K \cap B(x_0,r))^{N/(N-1)}.
    \end{equation}
    As $(V_i)_{i \geq 0}$ is a Caccioppoli partition of $\Omega$, it follows from \cite[Theorem 4.17]{AFP} that
    \begin{equation}\label{eq_ViKn}
        \sum_{i \geq 0} \HH^{N-1} \big(\partial^* V_i \cap B(x_0,r) \big) \leq 2 \HH^{N-1} \big(K \cap B(x_0,r) \big). 
    \end{equation} 
    By the ordering property of the sequence, we see that for all $i \geq 1$, we have
    \begin{equation*}
        \abs{B(x_0,r) \cap V_i} \leq \abs{B(x_0,r) \cap V_0} \leq \abs{B(x_0,r) \setminus V_i},
    \end{equation*}
    and thus by the relative isoperimetric inequality
    \begin{equation*}
        \abs{B(x_0,r) \cap V_i} \leq C \HH^{N-1}\big(B(x_0,r) \cap \partial^* V_i\big)^{N/(N-1)}.
    \end{equation*}
    We sum the inequality above over $i \geq 1$, we use the superadditivity of $t \mapsto t^{N/(N-1)}$, and we apply \eqref{eq_ViKn} to get the wanted estimate \eqref{eq_isoperimetric}.
    To conclude, we know by Lemma \ref{lem_af_volume} that
    \begin{equation*}
        \abs{B(x_0,r) \setminus V_0} \geq C^{-1} r^N
    \end{equation*}
    and thus (\ref{eq_isoperimetric}) implies
    \begin{equation*}
        \HH^{N-1}(K \cap B(x_0,r))^{N/(N-1)} \geq C^{-1} r^{N-1}.
    \end{equation*}
\end{proof}

\subsection{Uniform rectifiability}\label{section_UR}

A uniformly rectifiable set in $\R^N$ of dimension $N-1$ is a closed Ahlfors-regular set $E \subset \R^N$ which is contained in an image of a fairly nice parametrization $z : \R^{N-1} \to \R^N$. The following definition is extracted from \cite[Theorem 1.57]{DSBOOK}.

\begin{defi}
    Let $C \geq 1$ be a constant.
    We say that a closed set $E \subset \R^N$ is a uniformly rectifiable set of constant $C \geq 1$ if
    \begin{equation*}
        C^{-1} r^{N-1} \leq \HH^{N-1}(E \cap B(x_0,r)) \leq C r^{N-1} \quad \text{for all $x_0 \in E$, $r > 0$}
    \end{equation*}
    and $E \subset z(\R^N)$, where $z : \R^{N-1} \to \R^N$ is such that there exists a positive function $w \in L^1_{\mathrm{loc}}(\R^N)$ satisfying $\fint_B \omega \dd{x} \leq C \mathrm{ess. inf}_B \omega$ for all ball $B \subset \R^N$ and such that
    \begin{equation*}
        \abs{\nabla z} \leq \omega^{1/(N-1)} \quad \text{almost-everywhere,} 
    \end{equation*}
    and
    \begin{equation*}
        \int_{\set{y \in \R^d | z(y) \in B(x,r)}} \omega(y) \dd{y} \leq C r^{N-1} \quad \text{for all ball $B(x,r) \subset \R^N$.} 
    \end{equation*}
    The parametrization $z$ is called a $\omega$-regular parametrization and $\omega$ is called a $A_1$-weight.
\end{defi}

The definition allows bi-Lipschitz images of $\R^{N-1}$ into $\R^N$ but also surfaces with cusps and self-intersections to some extent. Whereas rectifiable sets are contained in a countable union of surfaces, a uniformly rectifiable set is contained in a single surface and the properties of the parametrization are meant to give quantitative information at all scales and locations (instead of in the blow-up regime almost-everywhere).

As in \cite{DS98}, we shall prove uniform rectifiable of quasiminimal sets using the following critera. It is proved in \cite{DS93}.
\begin{thm}[Condition B imply UR]\label{thm_BUR}
    Let $E \subset \R^N$ be a closed set and assume that there exists a constant $C \geq 1$ such that 
    \begin{enumerate}[label = (\roman*)]
        \item for all $x_0 \in E$, for all $r > 0$,
    \begin{equation*}
        C^{-1} r^{N-1} \leq \HH^{N-1}(E \cap B(x_0,r));
    \end{equation*}
\item $E$ satisfies condition B: for all $x_0 \in E$, for all $r > 0$, there exists two balls $B_1, B_2 \subset B(x_0,r) \setminus E$ with radius $\geq C^{-1} r$ and which lie in distinct connected components of $\R^N \setminus E$.
\end{enumerate}
    Then $E$ is a uniformly rectifiable set with a constant that depends only on $N$ and $C$.
\end{thm}

We now come back to the properties of quasiminimal sets.

\begin{prop}[Condition B]\label{prop_conditionB}
    Let $K$ be a $M$-quasiminimal set in $\Omega$. For all $x_0 \in K$ and $r > 0$ such that $B(x_0,2r) \subset \Omega$, there exists exists two balls $B_1, B_2 \subset B(x_0,r)$ with radius $\geq C^{-1} r$ which lie in distinct connected components of $\Omega \setminus K$.
\end{prop}
\begin{proof}
    We let $C \geq 1$ denote a generic constant which depends only on $N$ and $M$.
    Let $\varepsilon \in (0,1)$ and let $(y_i)_{i \in I}$ denote a maximal family of points $y_i \in K \cap B(x_0,3r/2)$ that lie at distance $\geq \varepsilon r$ from each other. Since the balls $B(y_i,\varepsilon r /2)$ are disjoint, contained $B(x_0,2r)$ and since $K$ is Ahlfors-regular, we have on the one hand
    \begin{equation*}
        \HH^{N-1}(K \cap \bigcup_i B(y_i,\varepsilon r/2)) \geq \sum_i \HH^{N-1}(K \cap B(y_i,\varepsilon r/2)) \geq C^{-1} m \varepsilon^{N-1} r^{N-1},
    \end{equation*}
    where $m$ is the number of element of $I$, and on the other hand
    \begin{equation*}
        \HH^{N-1}(K \cap \bigcup_i B(y_i,\varepsilon r/2)) \leq \HH^{N-1}(K \cap B(x_0,2r)) \leq C r^{N-1}.
    \end{equation*}
    It follows that $m \leq C \varepsilon^{1-N}$. For $t \in (r/2,r)$, we let $m(t)$ denote the number of $i \in I$ such that $\overline{B}(y_i,2\varepsilon r)$ meets $\partial B(x_0,t)$. Then,
    \begin{equation*}
        \int_{r/2}^r m(t) \dd{t} \leq C \int_{r/2}^r \sum_i \mathbf{1}_{t - \varepsilon r \leq \abs{y_i - x_0} \leq t + \varepsilon r} \dd{t} \leq C m \varepsilon r.
    \end{equation*}
    which allows to find a radius $t \in (r/2,r)$ such that $m(t) \leq C \varepsilon^{2-N}$. Setting
    \begin{equation*}
        Z = \partial B(x_0,t) \cap \left(\bigcup_i \overline{B}(y_i,2 \varepsilon r)\right),
    \end{equation*}
    we can thus estimate $\HH^{N-1}(Z) \leq C \varepsilon r^{N-1}$. Let us show that if $\varepsilon$ is small enough (depending on $N$ and $M$) then $\partial B(x,t) \setminus Z$ cannot be contained in a single connected component of $\Omega \setminus K$. We proceed by contradiction and observe that if this is the case, then
    \begin{equation*}
        F = K \cup Z \setminus B(x_0,t)
    \end{equation*}
    is a topological competitor of $K$ in $B(x_0,r)$ (this is the same argument as in the Proof of Lemma \ref{lem_separation}). This yields by quasiminimality (\ref{eq_quasi})
    \begin{equation*}
        \HH^{N-1}(K \cap B(x_0,t)) \leq M \HH^{N-1}(Z) \leq C \varepsilon r^{N-1},
    \end{equation*}
    and contredicts to Ahlfors-regularity (Proposition \ref{prop_af}) for $\varepsilon$ chosen sufficiently small (depending on $N$ and $M$).
    We conclude that $\partial B(x_0,t) \setminus Z$ meets at least two components of $\Omega \setminus K$. It is left to observe that for all $x \in \partial B(x_0,t)$, one has
    \begin{equation*}
        \mathrm{dist}(x,K) \geq \varepsilon r.
    \end{equation*}
    Indeed, if there exists $y \in K$ such that $\abs{x - y} < \varepsilon r$, then $y \in K \cap B(x_0,3r/2)$ and by maximality of the family $(y_i)_i$, there exists $i$ such that $\abs{y - y_i} < \varepsilon r$, whence $\abs{x - y_i} < 2\varepsilon r$. This contredicts the fact that $x \notin Z$.
\end{proof}

\begin{prop}[Uniform rectifiability]
    There exists a constant $C \geq 1$ depending only on $M$ and $N$ such that the following holds.
    Let $K$ be a $M$-quasiminimal set in $\Omega$. For all $x_0 \in K$ and $r > 0$ such that $B(x_0,2r) \subset \Omega$, there exists exists a uniformly rectifiable set $E \subset \R^N$ with constant $C$ such that $K \cap B(x_0,r) \subset E$.
\end{prop}
\begin{proof}
    Let $x_0 \in K$ and $r > 0$ be such that $B(x_0,2r) \subset \Omega$. Let $P$ by an hyperplane passing through $x_0$. Then the union $(K \cap \overline{B}(x_0,r)) \cup \partial B(x_0,r) \cup P$ satisfies the conditions of Theorem~\ref{thm_BUR}. We omit the details.
\end{proof}

Uniform rectifiability is not yet an optimal description of quasiminimal sets but is has many useful consequences, in particular it implies that the set is relatively flat in many balls. We are going to state this more precisely in Corollary \ref{cor_fl} but let us first recall the definiton of the flatness. For all $x_0 \in K$ and $r > 0$, the \emph{flatness} of $K$ in $B(x_0,r)$ is defined as
\begin{equation*}
    \beta_K(x_0,r) := r^{-1} \inf_P \sup_{y \in K \cap B(x_0,r)}{\rm dist }(y,P),
\end{equation*}
where the infimum is taken among all hyperplanes $P$ passing through $x_0$. The infimum is always attained for some hyperplan $P$ by compacity of the Grassmanian space $G(N-1,N)$. Thus, $\beta_K(x_0,r)$ is the smallest $\varepsilon > 0$ for which there exists an hyperplan $P$ passing through $x_0$ such that
\begin{equation*}
    K \cap B(x_0,r) \subset \set{\mathrm{dist}(\cdot,P) \leq \varepsilon r}.
\end{equation*}
When there is no ambiguity, we write $\beta(x_0,r)$ instead of $\beta_K(x_0,r)$.
A standard consequence of rectifiability and Ahlfors-regularity is that for $\HH^{N-1}$-a.e. $x_0 \in K$, we have
\begin{equation}\label{eq_limit_beta}
    \lim_{r \to 0} \beta_K(x_0,r) = 0.
\end{equation}
Quasiminimal sets satisfy a more quantitative variant of (\ref{eq_limit_beta}), called the Weak Geometric Lemma.
\begin{lem}[Weak Geometric Lemma]
    Let $K$ be a $M$-quasiminimal set in $\Omega$.
    For all $\varepsilon > 0$, there exists a constant $C = C(\tau) \geq 1$ (which depends only on $\varepsilon$, $N$, $M$) such that for all $x_0 \in K$ and for all $r > 0$ with $B(x_0,2r) \subset \Omega$, we have
    \begin{equation*}
        \int_{K \cap B(x_0,r)} \int_0^{r} \mathbf{1}_{\set{(y,t) | \beta_K(y,t) \geq \varepsilon}}(y,t) \frac{\dd{t}}{t} \dd{\HH^{N-1}(y)} \leq C r^{N-1}.
    \end{equation*}
\end{lem}
This means that there are many balls where $\beta_K(y,t) < \varepsilon$. The Weak Geometric Lemma has the following consequence: in any ball, one can find a smaller ball (but not too much smaller) with a shifted center where the flatness is small.
\begin{cor}\label{cor_fl}
    Let $K$ be a $M$-quasiminimal set in $\Omega$.
    For all $\varepsilon > 0$, there exists a constant $C = C(\varepsilon) \geq 1$ (which depends only on $N$, $M$ and $\varepsilon$) such that for all $x_0 \in K$ and for all $r > 0$ with $B(x_0,r) \subset \Omega$, there exists $y \in K \cap B(x_0,r/2)$ and $t \in (C^{-1}r,r/2)$ satisfying
    \begin{equation*}
        \beta_K(y,t) \leq \varepsilon.
    \end{equation*}
\end{cor}
\begin{proof}
    This is proved by contradiction using the fact that
    \begin{equation*}
        \int_{K \cap B(x_0,r/2)} \int_{0}^{r/2} \mathbf{1}_{\beta(y,t) \geq \varepsilon}(y,t) \frac{\dd{t}}{t} \dd{\HH^{N-1}(y)} \leq C_0 r^{N-1}.
    \end{equation*}
    and that $\HH^{N-1}(K \cap B(x_0,r/2)) \geq C^{-1} r^{N-1}$.
\end{proof}

We conclude this section by stating a consequence of Proposition \ref{prop_conditionB} which will be useful in later sections.
\begin{cor}\label{lem_plane_separation}
    There exists a constant $\varepsilon_0 \in (0,1/2)$ which depends on $N$, $M$ such that the following property holds true.
    Let $K$ be a $M$-quasiminimal set in $\Omega$. Let $x_0 \in K$, $r > 0$ be such that $B(x_0,r) \subset \Omega$ and for which there exists some hyperplane $P$ passing through $x_0$ such that
    \begin{equation}\label{eq_plane_separation}
        K \cap B(x_0,r) \subset \set{x \in B(x_0,r) | \mathrm{dist}(x,P) \leq \varepsilon_0 r}.
    \end{equation}
    Then the two components of $\set{x \in B(x_0,r) | \mathrm{dist}(x,P) > \varepsilon_0 r}$ lie in in distinct connected components of $\Omega \setminus K$.
\end{cor}
\begin{figure}[ht]
\centerline{\includegraphics[width=0.37\linewidth]{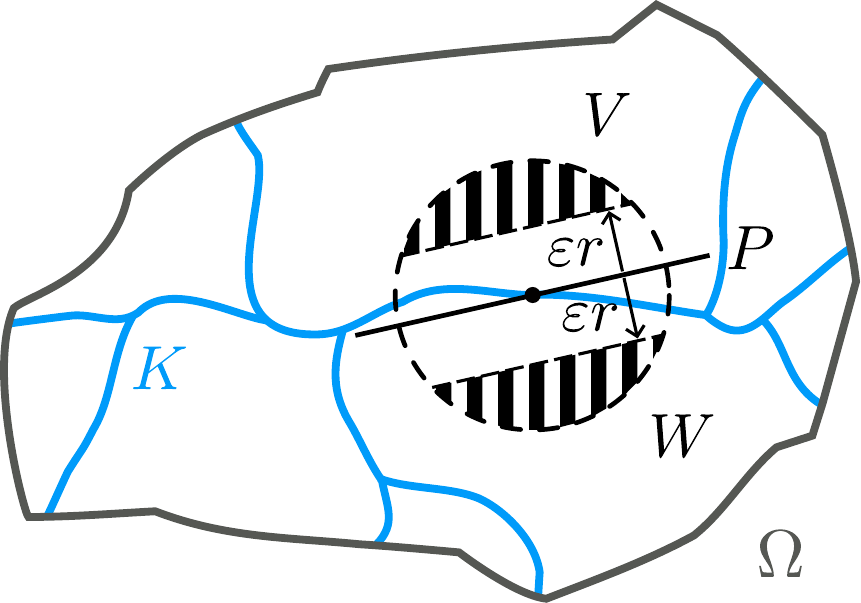}}
\caption{An illustration for Corollary~\ref{lem_plane_separation}.}
\end{figure}
\begin{proof}
    We let $C \geq 1$ denote a generic constant which depends only on $N$ and $M$.
    According to Proposition \ref{prop_conditionB}, there exists two balls $B_1, B_2 \subset B(x_0,r/2)$ which are of radius $\geq C^{-1} r$ and which lie in distinct connected component of $\Omega \setminus K$, say $V$ and $W$. Taking $\varepsilon_0$ small enough in (\ref{eq_plane_separation}), none of these two balls can be contained in the strip
    \begin{equation*}
        \set{x \in B(x_0,r) | \mathrm{dist}(x,P) \leq \varepsilon_0 r}.
    \end{equation*}
    Since the two components $\set{x \in B(x_0,r) | \mathrm{dist}(x,P) > \varepsilon_0 r}$ are connected subset of $\Omega \setminus K$, one must be in $V$ and the other in $W$.

\end{proof}

\subsection{No isolated components}

\begin{prop}\label{prop_isolated}
    Let $K$ be a $M$-quasiminimal set in $\Omega$. For all $x_0 \in K$ and for all $r > 0$ such that $B(x_0,r) \subset \subset \Omega$, there is not connected component $V$ of $\Omega \setminus K$ such that $V \subset B(x_0,r)$.
\end{prop}
\begin{proof}
    Let us assume that there exists such a connected component $V$. We know that $V$ is a set of finite perimeter with $\partial^* V \subset K$ and that there exists a universal constant $C_0 \geq 1$ such that for $\HH^{N-1}$-a.e. $x \in \partial^* V$, we have
    \begin{equation*}
        \liminf_{t \to 0} \abs{V \cap B(x,r)} \geq C_0^{-1} t^N.
    \end{equation*}
    We also know by rectifiability and Ahlfors-regularity of $K$ that for $\HH^{N-1}$-a.e. $x \in K$, we have
    \begin{equation*}
        \lim_{t \to 0} \beta(x,t) = 0.
    \end{equation*}
    We let $\varepsilon > 0$ which will be chosen small enough later (depending on $N$ and $M$).
    As $V$ is a non-empty open set such that $V \subset \subset B(x_0,r)$, we have $\HH^{N-1}(\partial^* V \cap B(x_0,t)) > 0$. This allows to select a point $x_1 \in \partial V$ and a small radius $t > 0$ such that $B(x_1,t) \subset B(x_0,r)$,
    \begin{equation*}
        \abs{V \cap B(x_1,t)} \geq C_0^{-1} t^N
    \end{equation*}
    and such that there exists an hyperplane $P$ passing through $x_1$ satisfying
    \begin{equation*}
        K \cap \overline{B}(x_1,t) \subset \overline{B}(x_1,t) \cap \set{\mathrm{dist}(\cdot,P) \leq \varepsilon t}.
    \end{equation*}
    We choose $\varepsilon$ small enough (depending on $N$) so that
    \begin{equation*}
        \abs{\overline{B}(x_1,t) \cap \set{\mathrm{dist}(\cdot,P) \leq \varepsilon r}} < C_0^{-1} t^n.
    \end{equation*}
    Thus, it is not possible that
    \begin{equation*}
        V \cap B(x_1,t) \subset \set{\mathrm{dist}(\cdot,P) \leq \varepsilon t}
    \end{equation*}
    and this guarantees that one of the components of $\overline{B}(x_1,t) \cap \set{\mathrm{dist}(\cdot,P) > \varepsilon t}$ must be contained in $V$.
    Now, we set
    \begin{equation*}
        F = \left(K \setminus B(x_1,t)\right) \cup Z,
    \end{equation*}
    where
    \begin{equation*}
        Z := \set{x \in \partial B(x_1,t) | \mathrm{dist}(y,P) \leq \varepsilon t}.
    \end{equation*}
    and we check that $F$ is a topological competitor of $K$ in $B(x_0,r)$ (not $B(x_1,t)$).
\begin{figure}[ht]
\begin{center}
\includegraphics[width=0.28\linewidth]{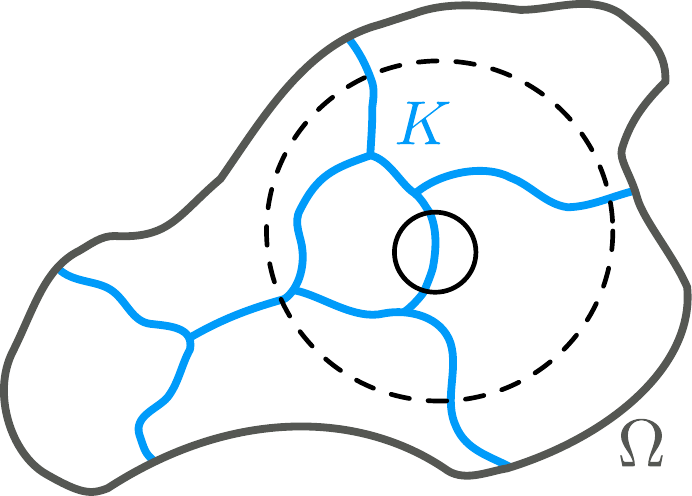}
\quad \quad
\includegraphics[width=0.28\linewidth]{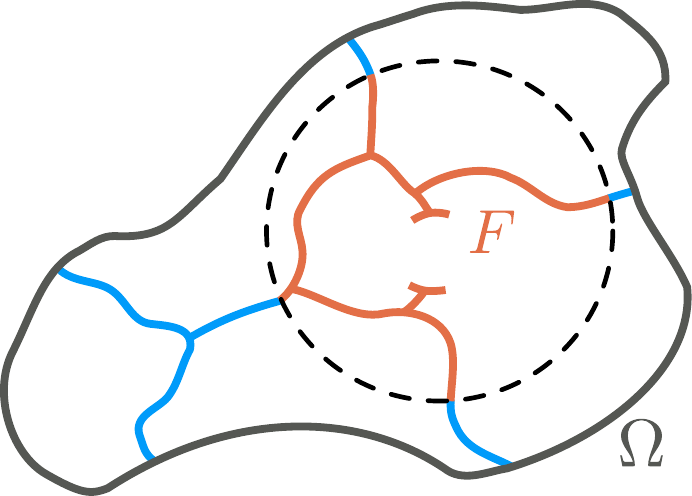}
\end{center}
\caption{The competitor in Proposition~\ref{prop_isolated}.}
\end{figure}
    We consider two points $x$ and $y$ in $\Omega \setminus (\overline{B}(x_0,r) \cup K)$ which are connected by a path in $\Omega \setminus F$ and we want to show that $x$ and $y$ are connected in $\Omega \setminus K$.
    If the path does not meet $\partial B(x_1,t)$, then it stays outside of $\overline{B}(x_1,t)$, where $F$ and $K$ coincide so $x$ and $y$ are connected in $\Omega \setminus K$. If the path meets $\partial B(x_1,t)$, it can only be at a point of $\partial B(x_1,t) \setminus Z$
    We observe that $\partial B(x_1,t) \setminus Z$ is composed of two spherical caps, denoted by $Z_1$, $Z_2$, which are disjoint from $K$.
    We consider the portion of the path which starts from $x$ and meets $\partial B(x_1,t)$ for the first time. This portion is disjoint from $K$ so $Z_1$ must be contained in the same connected component of $\Omega \setminus K$ as $x$. As $x \notin V$, then $Z_1$ is contained in a connected component of $\Omega \setminus K$ which differs from $V$ and this forces $Z_2 \subset V$.
    Similarly, we consider the portion of the path which leaves $\partial B(x_1,t)$ for the last time and arrives at $y$. This portion is disjoint from $K$ and as $y \notin V$, it connects necessarily $y$ to $Z_1$. We conclude that $x$ and $y$ are connected in $\Omega \setminus K$ à $Z_1$ and as $Z_1$ is a connect dpart of $\Omega \setminus K$, the points $x$ and $y$ area also connected in $\Omega \setminus K$. We conclude that $F$ is a topological competitor of $K$ in $B(x_0,r)$.
    We finally apply the quasiminimality property (\ref{eq_quasi}) in $B(x_0,r)$,
    \begin{equation*}
        \HH^{N-1}(K \setminus F) \leq M \HH^{N-1}(F \setminus K),
    \end{equation*}
    so
    \begin{equation*}
        \HH^{N-1}(K \cap B(x_1,t)) \leq M \HH^{N-1}(Z) \leq C M \varepsilon t^{N-1}.
    \end{equation*}
    This contradicts the Ahlfors-regularity of $K$ if $\varepsilon_0$ is chosen small enough.
\end{proof}

\section{Interface points}\label{section_generic}

\begin{defi}
    Let $K$ be a $M$-quasiminimal set in $\Omega$.
    We define the set of \emph{interface points} $K^*$ as the set of points $x_0 \in K$ for which there exists a radius $r > 0$ such that $B(x_0,r) \subset \Omega$ and $B(x_0,r) \setminus K$ meets exactly two components of $\Omega \setminus K$.
\end{defi}
When $B(x_0,r) \setminus K$ meets exactly two components of $\Omega \setminus K$, Lemma \ref{lem_separation} shows that this is also the case of $B(x_0,t) \setminus K$ for all $0 < t < r$. We deduce that for all $x \in K^*$, there exists a unique pair of components $V$, $W$ of $\Omega \setminus K$ such that $x \in \partial V \cap \partial W$. One can justify similarly that $K^*$ is a relative open subset of $K$.
The main goal of this section is to prove the following Proposition.
\begin{prop}\label{prop_generic_points}
    Let $K$ be a $M$-quasiminimal set in $\Omega$, then $\HH^{N-1}(K \setminus K^*) = 0$ and for all component $V$ of $\Omega \setminus K$,
    \begin{equation*}
        \HH^{N-1}(\Omega \cap \partial V \setminus \partial^* V) = 0. 
    \end{equation*}
\end{prop}


In order to prove Proposition \ref{prop_generic_points}, we start with a Lemma in the same spirit as Lemma \ref{lem_af_volume}. It says that if two components $V$, $W$ of $\Omega \setminus K$ fill most of the volume of a ball $B(x_0,r) \setminus K$, there cannot be an inflitration of another component of $\Omega \setminus K$ which meets $B(x_0,r/2)$. Results of this type are standard in the literature \cite{Tamanini91, Tamanini96, White, Leonardi, Maggi} and the proof consists in building a competitor by merging the infiltration in $V$ or in $W$. Our setting differs though because our objects are not Caccioppoli partitions but closed sets; we should build the competitor through set operations on the topological boundaries of $V$, $W$ and $(V \cup W)^c$ instead of their reduced boundaries. This poses a difficulty in the case where they share a common and non-negligible piece of boundary, similarly to the lakes of Wada.

\begin{lem}\label{lem_infiltration}
    There exists a constant $\varepsilon_0 \in (0,1)$ which depends on $N$, $M$ such that the following property holds true.
    Let $K$ be a $M$-quasiminimal set in $\Omega$. Then for all $x_0 \in K$ and for all $r > 0$ such that $B(x_0,r) \subset \Omega$ and for all distinct connected components $V$, $W$ of $\Omega \setminus K$, if
    \begin{equation*}
        \abs{B(x_0,r) \setminus (V \cup W)} \leq \varepsilon_0 r^N, 
    \end{equation*}
    then $B(x_0,r/2) \setminus K$ only meets $V$ and $W$.
\end{lem}
\begin{proof}
    We let $x_0 \in K$ and $r > 0$ be such that $B(x_0,r) \subset \Omega$ and $V$, $W$ be distinct connected components of $\Omega \setminus K$ such that
    \begin{equation}\label{eq_initVW}
        \abs{B(x_0,r) \setminus (V \cup W)} \leq \varepsilon_0 r^N,
    \end{equation}
    for some small $\varepsilon_0 > 0$ which will be fixed later (depending on $N$ and $M$).
    According to the co-area formula,
    \begin{equation*}
        \abs{B(x_0,r) \setminus (V \cup W)} = \int_0^r \HH^{N-1}(\partial B(x_0,t) \setminus (V \cup W)) \dd{t}
    \end{equation*}
    so we can find a radius $t \in (r/2,r)$ such that
    \begin{equation}\label{eq_tVW}
        \HH^{N-1}(\partial B(x_0,t) \setminus (V \cup W)) \leq C r^{-1} \abs{B(x_0,r) \setminus (V \cup W)} \leq C \varepsilon_0 r^{N-1}.
    \end{equation}
    The next step of the proof consists in showing that
    \begin{equation}\label{eq_goal_VW}
        \HH^{N-1}(B(x_0,t) \cap \partial (V \cup W)) \leq C \HH^{N-1}(\partial B(x_0,t) \setminus (V \cup W)).
    \end{equation}
    The principle is to build a competitor by removing the piece $B(x_0,t) \cap \partial V \cap \partial (V \cup W)$ or $B(x_0,t) \cap \partial W \cap \partial (V \cup W)$ (depending on which choice removes the most area) but adding $\partial B(x_0,t) \setminus (V \cup W)$ (for the topological condition (\ref{eq_topological}) to be satisfied). Then (\ref{eq_goal_VW}) follows by quasiminimality and the fact that
    \begin{multline*}
    \HH^{N-1}(B(x_0,t) \cap \partial (V \cup W)) \leq \HH^{N-1}(B(x_0,t) \cap \partial V \cap \partial (V \cup W)) \\+ \HH^{N-1}(B(x_0,t) \cap \partial W \cap \partial (V \cup W)).
    \end{multline*}
    The issue here is that admissible competitors must be closed sets so we can only remove a relative open subset of $K$. A first try is to take $B(x_0,t) \cap \partial V \setminus \partial W$ as an open replacement of $B(x_0,t) \cap \partial V \cap \partial (V \cup W)$. However, the former might have a strictly smaller area in case the set of ``triple points'' $\partial V \cap \partial W \cap \partial (V \cup W)$ is not $\HH^{N-1}$-neglible. We will instead build the competitor via a more technical covering argument.

    We consider some $\varepsilon > 0$ and a family of balls $(B(y_k,t_k))_{k}$, where $y_k \in B(x_0,t) \cap \partial^* (V \cup W)$ and $t_k > 0$, such that the closed balls $(\overline{B}(y_k,t_k))_k$ are disjoint,
    \begin{equation}\label{eq_partialVW}
        \HH^{N-1}\Bigl(B(x_0,t) \cap \partial^* (V \cup W) \setminus \bigcup_i \overline{B}(y_k,t_k)\Bigr) = 0
    \end{equation}
    and for all $k$,
    \begin{enumerate}[label = (\roman*)]
    \item $\overline{B}(y_k,t_k) \subset B(x_0,t)$;
    \item there exists an hyperplane $P_k$ passing through $y_k$ such that
    \begin{equation*}
        K \cap \overline{B}(y_k,t_k) \subset \set{\mathrm{dist}(\cdot,P_k) \leq \varepsilon t_k};
    \end{equation*}
    \item
    \begin{equation}\label{eq_Hvolume}
        \left(\frac{1}{2} - \frac{1}{100}\right) \abs{B(y_k,t_k)} \leq \abs{B(y_k,t_k) \cap (V \cup W)} \leq \left(\frac{1}{2} + \frac{1}{100}\right) \abs{B(y_k,t_k)}.
    \end{equation}
\end{enumerate}
    Here, note that $\Omega \cap \partial^* (V \cup W) \subset K$ so the balls are centered in $K$.
    The existence of such a family of balls can be justified using the Vitali covering Theorem \cite[Theorem 2.8]{mattila} with respect to the Radon measure $\mu = \HH^{N-1} \mres K$ and using also the standard properties of reduced boundaries \cite[Theorem 3.61]{AFP}.
    As $\HH^{N-1}\left(B(x_0,t) \cap \partial^* (V \cup W)\right) < +\infty$, we can also assume that the family of balls is finite, provided that we replace (\ref{eq_partialVW}) by
        \begin{equation}\label{eq_partialVW2}
        \HH^{N-1}\Bigl(B(x_0,t) \cap \partial^* (V \cup W) \setminus \bigcup_k B(y_k,t_k)\Bigl) \leq \varepsilon.
    \end{equation}
    If $\varepsilon$ is small enough as in Lemma \ref{lem_plane_separation}, the connected components of 
    \begin{equation}\label{eq_VWP}
        \set{x \in \overline{B}(y_k,t_k) | \mathrm{dist}(x,P_k) > \varepsilon t_k}
    \end{equation}
    lie in distinct connected components of $\Omega \setminus K$.
    In view of (\ref{eq_Hvolume}), it is not possible for $B(y_k,t_k) \cap (V \cup W)$ neither $B(y_k,t_k) \setminus (V \cup W)$ to be contained in $\set{\mathrm{dist}(\cdot,P_k) \leq \varepsilon t_k}$ if $\varepsilon$ is small enough because otherwise $B(y_k,t_k) \cap (V \cup W)$ would have a too small or a too big measure. Therefore, one of the components of (\ref{eq_VWP}) must be contained in $V$ or $W$, and the other component of (\ref{eq_VWP}) must be contained in a component of $\Omega \setminus K$ which is neither $V$, neither $W$ and which we denote by $U_k$. We let $S(V)$ denote the set of indices $k$ such that one of the component of (\ref{eq_VWP}) is contained in $V$ (resp. $S(W)$ for $W$).
    By (\ref{eq_partialVW2}) and Ahlfors-regularity of $K$, we have
    \begin{align}
        \HH^{N-1}(B(x_0,t) \cap \partial^* (V \cup W)) &\leq \HH^{N-1}(K \cap \bigcup_k \overline{B}(y_k,t_k)) + \varepsilon\nonumber\\
                                                       &\leq C \sum_{k \in S(V)} t_k^{N-1} + \sum_{k \in S(W)} t_k^{N-1} + \varepsilon.\label{eq_VW_select}
    \end{align}
    Without loss of generality, we assume that
    \begin{equation}\label{eq_VW_select2}
        \HH^{N-1}(B(x_0,t) \cap \partial^* (V \cup W)) \leq C \sum_{k \in S(V)} t_k^{N-1} + \varepsilon,
    \end{equation}
    where the constant $C$ has only be multiplied by $2$ compared to (\ref{eq_VW_select}). In the construction below, we build a competitor by making holes in the balls indexed by $k \in S(V)$ and this this meant to remove most of $B(x_0,t) \cap \partial V \cap \partial^* (V \cup W)$.
    For all $k \in S(V)$, we let
    \begin{equation*}
        Z_k := \set{x \in \partial B(y_k,t_k) | \mathrm{dist}(x,P_k) \leq \varepsilon t_k}
    \end{equation*}
    and we define
    \begin{equation*}
        F := \Bigl(K \setminus \bigcup_{l \in S(V)} B(y_k,t_k)\Bigr) \cup \bigcup_{k \in S(V)} Z_k \cup \Bigl(\partial B(x_0,t) \setminus (V \cup W)\Bigr).
    \end{equation*}
    We then justify that $F$ is a topological competitor of $K$ in $\overline{B}(x_0,t)$.
    It is clear that $F$ is relatively closed in $\Omega$ because it is a finite union of relatively closet set.
    Let $x$ and $y$ be two points in $\Omega \setminus (\overline{B}(x_0,t) \cup K)$ which are connected by a path in $\Omega \setminus F$. We want to show that $x$ and $y$ lie in the same component of $\Omega \setminus K$.
    If the path does not intersect $\partial B(x_0,t)$, then it stays in the complement of $\overline{B}(x_0,t)$, where $F$ coincide with $K$. In this case, the path does not meet $K$ and the points $x$ and $y$ are also connected in $\Omega \setminus K$. If the path meets $\partial B(x_0,t)$, it can only be at a point of $V \cup W$ because $F$ contains $\partial B(x_0,t) \setminus (V \cup W)$. Considering the portion of the path leaving $x$ until the first time it meets $\partial B(x_0,t)$, we see that $x$ is connected to $V$ or $W$ in the complement of $K$, and similarly for $y$.
    Assuming by contradiction that the two points $x$, $y$ are not in the same connected components of $\Omega \setminus K$, then at least one of them is in $W$, let's say $x$. We consider the portion of the path which starts for $x$ and meets $K$ for the first time. As the path is disjoint from $K \setminus \bigcup \set{B(y_k,t_k) | k \in S(V)} \subset F$, the meeting point with $K$ can only be at a point of $B(y_k,t_k) \cap K$. In particular, the path must cross $\partial B(y_k,t_k)$ before meeting $K$ and since it is disjoint from $Z_k \subset F$, it can only be at a point of $V$ or $U_k$.
    This contradicts the fact that before meeting $K$ for the first time, the path must be contained in $W$.
    We conclude that $F$ is a topological competitor in $\overline{B}(x_0,t)$.

    We finally apply the quasiminimality property (\ref{eq_quasi}) in $\overline{B}(x_0,t)$,
    \begin{equation*}
        \HH^{N-1}(K \setminus F) \leq M \HH^{N-1}(F \setminus K),
    \end{equation*}
    whence
    \begin{align*}
        \HH^{N-1}(K \cap \bigcup_{k \in S(V)} B(y_k,t_k)) &\leq M \HH^{N-1}(\partial B(x_0,t) \setminus (V \cup W)) + M \sum_{k \in S(V)} \HH^{N-1}(Z_k)\\
                                                          &\leq M \HH^{N-1}(\partial B(x_0,t) \setminus (V \cup W)) + C \varepsilon \sum_{k \in S(V)} t_k^{N-1}.
    \end{align*}
    By (\ref{eq_VW_select2}) and since the balls $(B(y_k,t_k))_k$ are disjoint, centered on $K$ and $K$ is Ahlfors-regular, we have
    \begin{align*}
        \HH^{N-1}(B(x_0,t) \cap \partial^* (V \cup W)) &\leq C \sum_{k \in S(V)} t_k^{N-1} + \varepsilon\\
                                                       &\leq \HH^{N-1}(K \cap \bigcup_{k \in S(V)} B(y_k,t_k)) + \varepsilon.
    \end{align*}
    We can bound from above similarly
    \begin{equation*}
        C \varepsilon \sum_{k \in S(V)} t_k^{N-1} \leq C \varepsilon \HH^{N-1}(B(x_0,t) \cap K).
    \end{equation*}
    but the important point is that this term goes to $0$ when $\varepsilon \to 0$. We conclude that 
    \begin{equation}\label{eq_step1_VW}
        \HH^{N-1}(B(x_0,t) \cap \partial^* (V \cup W)) \leq C \HH^{N-1}(\partial B(x_0,t) \setminus (V \cup W)),
    \end{equation}
    where $C \geq 1$ depends on $N$ and $M$.

    Next, by (\ref{eq_step1_VW}) and (\ref{eq_tVW}), we find
    \begin{equation*}
        \HH^{N-1}(B(x_0,t) \cap \partial^* (V \cup W)) \leq C \varepsilon_0 r^{N-1}.
    \end{equation*}
    Using (\ref{eq_initVW}), we can choose $\varepsilon_0$ small enough (depending only on $N$) so that the small volume of $B(x_0,r) \setminus (V \cup W)$ implies
    \begin{equation*}
        \abs{B(x_0,r/2) \setminus (V \cup W)} \leq \tfrac{1}{2} \abs{B(x_0,r/2)},
    \end{equation*}
    and thus we can apply the relative isoperimetric inequality to $B(x_0,r/2) \setminus (V \cup W)$ in the ball $B(x_0,r/2)$. This yields
    \begin{align*}
        \abs{B(x_0,r/2) \setminus (V \cup W)} &\leq C \HH^{N-1}(B(x_0,r/2) \cap \partial^* (V \cup W))^{N/(N-1)}\\
                                              &\leq C \varepsilon_0^{N/(N-1)} r^N.
    \end{align*}
    We can thus choose $\varepsilon_0$ small enough again (depending only on $N$ and $M$) so that
    \begin{equation*}
        \abs{B(x_0,r) \setminus (V \cup W)} \leq \varepsilon_0 r^N \implies \abs{B(x_0,r/2) \setminus (V \cup W)} \leq \varepsilon_0 (r/2)^N.
    \end{equation*}
    Iterating this estimate, we deduce
    \begin{equation}\label{eq_conclusion_VW}
        \abs{B(x_0,r) \setminus (V \cup W)} \leq \varepsilon_0 r^N \implies \lim_{t \to 0} t^{-N} \abs{B(x_0,t) \setminus (V \cup W)} \leq 2^{N} \varepsilon_0.
    \end{equation}
    Let us assume now that we have $\abs{B(x_0,r) \setminus (V \cup W)} \leq \varepsilon_0 (r/2)^N$, where $\varepsilon_0$ is small enough for (\ref{eq_conclusion_VW}) to hold.
    We see that for all $x \in K \cap B(x_0,r/2)$, we have
    \begin{equation*}
        \abs{B(x,r/2) \setminus (V \cup W)} \leq \varepsilon_0 (r/2)^N,
    \end{equation*}
    and then by (\ref{eq_conclusion_VW}) that for all $x \in K \cap B(x_0,r/2)$,
    \begin{equation*}
        \lim_{t \to 0} t^{-N} \abs{B(x,t) \setminus (V \cup W)} \leq 2^N \varepsilon_0.
    \end{equation*}
    This holds true in particular for all points $x \in B(x_0,r/2) \cap \partial^* (V \cup W)$ but, if $\varepsilon_0$ is chosen small enough (depending only on $N$), this contredicts the properties of the reduced boundary (\cite[Theorem 3.61]{AFP}) so $\HH^{N-1}(B(x_0,r/2) \cap \partial^* (V \cup W)) = 0$. Then the relative isoperimetric inequality in $B(x_0,r/2)$ shows that we must have $\abs{B(x_0,r/2) \cap \partial^* (V \cup W)} = 0$ or $\abs{B(x_0,r/2) \setminus (V \cup W)} = 0$ but the former is impossible due to our initial assumption (\ref{eq_initVW}) with a small enough $\varepsilon_0$. We conclude that $B(x_0,r/2)$ cannot meet a component of $\Omega \setminus K$ other than $V$ or $W$.
\end{proof}

We deduce a sufficient condition for a point $x_0 \in K$ to be in $K^*$.
\begin{prop}\label{prop_plane_separation}
    There exists a constant $\varepsilon_0 \in (0,1)$ which depends on $N$, $M$ such that the following property holds true.
    Let $K$ be a $M$-quasiminimal set in $\Omega$. For all $x_0 \in K$ and $r > 0$ such that $B(x_0,r) \subset \Omega$ and $\beta(x_0,r) \leq \varepsilon_0$, 
    \begin{equation*}
        \text{$B(x_0,r/2)$ meets exactly two components of $\Omega \setminus K$.}
    \end{equation*}
\end{prop}
\begin{proof}
    Let $P$ be an hyperplane passing through $x_0$ and which atteins the minimum in the definition of $\beta(x_0,r)$.
    We know by Lemma \ref{lem_plane_separation} that if $\varepsilon_0$ is small enough (depending on $N$ and $M$), then the two components of
    \begin{equation*}
        \set{x \in B(x_0,r) | \mathrm{dist}(x,P) > \varepsilon r}
    \end{equation*}
    are contained in distinct components of $\Omega \setminus K$. If $\varepsilon_0$ is chosen small enough once again, Lemma \ref{lem_infiltration} shows that that there cannot be any other component of $\Omega \setminus K$ which meets $B(x_0,r/2)$, and thus, $x \in K^*$.
\end{proof}

We are now ready to justify Proposition \ref{prop_generic_points}.
\begin{proof}[Proof of Proposition \ref{prop_generic_points}]
    It is standard by Ahlfors-regularity and rectifiability that for $\HH^{N-1}$-a.e. $x \in K$, we have $\lim_{r \to 0} \beta(x,r) = 0$, and according to Proposition \ref{prop_plane_separation}, such a point belongs to $K^*$. If $V$ is a component of $\Omega \setminus K$, then for $\HH^{N-1}$-a.e. $x \in \Omega \cap \partial V$, we have $\lim_{r \to 0} \beta(x,r) = 0$ and, from the proof of Proposition \ref{prop_plane_separation}, one can see that this implies
    \begin{equation*}
        \lim_{r \to 0} \frac{\abs{V \cap B(x,r)}}{\abs{B(x,r)}} = \frac{1}{2}.
    \end{equation*}
    This property characterizes the reduced boundary $\partial^* V$ (up to a $\HH^{N-1}$-negligible set), see \cite[Theorem 3.61]{AFP}.
\end{proof}

It will now be easier to build competitor by merging components into an other; the covering argument of Lemma \ref{lem_infiltration} won't be needed anymore.
\begin{rmk}[Merging components into others]\label{rmk_merging}
    Let $B(x_0,r) \subset \Omega$ be an open ball with center $x_0 \in K$ and radius $r > 0$.
    Let $(V_i)_{i \in I}$ denote the connected components of $\Omega \setminus K$. We select one of them, say $V_k$, and a selection of other components $(V_j)_{j \in J}$ with $J \subset I \setminus \set{k}$.
    The competitor obtained by \emph{merging $(V_j)_{j \in J}$ into $V_k$ within $B$} is defined as
    \begin{equation*}
        F :=  \Bigl(\partial B \cap \bigcup_{j \in J} V_j\Bigr) \cup K \setminus \Bigl(B \cap K^* \cap \bigcup_{i, j \in J \cup \set{k}} (\partial V_i \cap \partial V_j)\Bigr).
    \end{equation*}
    This is a relatively closed subset of $\Omega$ as $K^* \cap \partial V \cap \partial V_i$ and $\Omega \cap \partial (\bigcup_{j \in J} V_j) \subset K$.
    We will check soon $F$ is a topological competitor of $K$ in $\overline{B}$ but let us already state what quasiminimality property says in this case. The inequality
    \begin{equation*}
        \HH^{N-1}(K \setminus F) \leq C \HH^{N-1}(F \setminus K) 
    \end{equation*}
    entails
    \begin{equation}\label{eq_merging}
        \HH^{N-1}\Bigl(B \cap \bigcup_{i,j \in J \cup \set{k}} \partial V_i \cap \partial V_j\Bigr) \leq C \sum_{j \in J} \HH^{N-1}\Bigl(\partial B \cap V_j\Bigr).
    \end{equation}
    Property (\ref{eq_merging}) will be instrumental in the next sections and we shall refer to it as the inequality obtained by \emph{merging $(V_j)_{j \in J}$ into $V_k$ within $B$}.

\begin{figure}[ht]
\begin{center}
\includegraphics[width=0.34\linewidth]{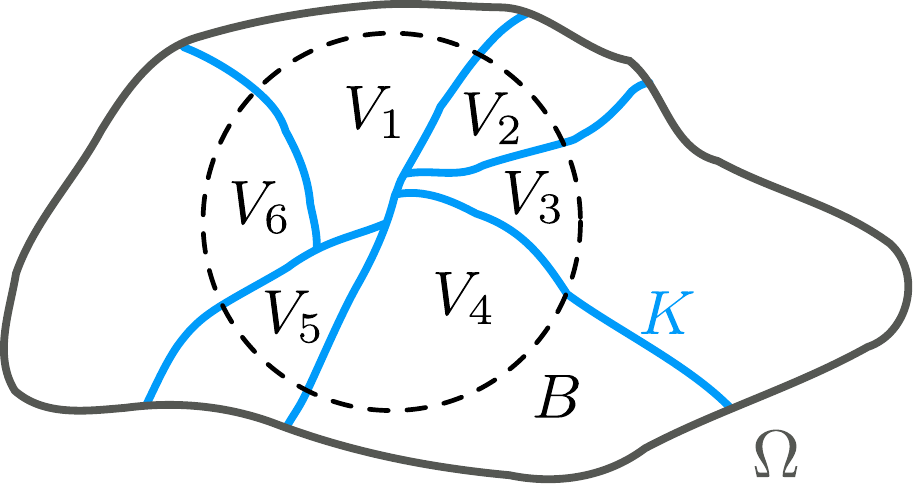}
\quad \quad
\includegraphics[width=0.34\linewidth]{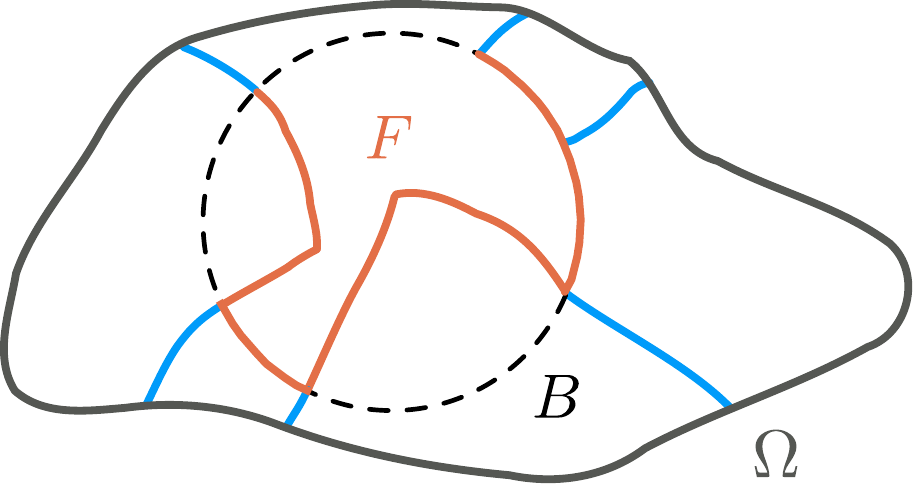}
\end{center}
\caption{Merging $V_2, V_3, V_5$ into $V_1$ within $B$.}
\end{figure}

    We finally justify the validity of (\ref{eq_merging}) by showing that $F$ is a topological competitor of $K$ in $B$.
    We consider two points $x$ and $y$ in $\Omega \setminus (\overline{B} \cup K)$ which are connected in $\Omega \setminus F$ and we want to show that $x$ and $y$ are connected in $\Omega \setminus K$.
    We proceed by contradiction and assume that $x$ and $y$ are not connected in $\Omega \setminus K$.
    If the path connecting $x$ and $y$ does not intersect $\partial B$, then it stays in the complement of $\overline{B}$ and does not meet $K$, which is a contradiction. Otherwise, the path meets $\partial B$ and it can only be at a point of $\partial B \setminus \left(K \cup \bigcup_{j \in J} V_j\right)$ because $F$ contains $\partial B \cap K$ and $\partial B \cap \bigcup_{j \in J} V_j$. Considering the portion of the path leaving $x$ until the first time it meets $\partial B$, we see that $x$ belongs to some connected component $W_1$ of $\Omega \setminus K$ which is distinct from the $V_j$ for al $j \in J$. Similarly, $y$ belongs to a connected component $W_2$ of $\Omega \setminus K$ which is distinct from the $V_j$ for all $j \in J$.
    If $W_1$ differs from $W_2$, then at least one of them, say $W_1$, is distinct from $V_k$. We consider the portion of the path which starts from $x$ and meets $K$ for the first time. 
    As the path is disjoint from
    \begin{equation*}
        K \setminus (B \cap K^* \cap \bigcup_{i, j \in J \cup \set{k}} \partial V_i \cap \partial V_j)
    \end{equation*}
    it can only meets $K$ at a point of $B \cap K^* \cap \bigcup_{i, j \in J \cup \set{k}} \partial V_i \cap \partial V_j$. Therefore, there exists a small $r > 0$ such that $B(x,r) \setminus K$ is covered by exactly two components $V_i$ and $V_j$, where $i$ and $j$ belong to $J \cup \set{k}$. In particular, this point cannot be in $\partial W$ but this contradicts the fact that before meeting $K$, the path was contained in $W$. We conclude that $W_1 = W_2$ so $x$ and $y$ are connected in $\Omega \setminus K$.
    This is again a contradiction to the assumption, and proves that $F$ is a topological competitor of $K$ in $\overline{B}$.
\end{rmk}

\section{Local finiteness}\label{section_finiteness}

\subsection{In the plane}

In the plane, the local finiteness is standard and not difficult to prove. This is based on the observation that the components of $\Omega \setminus K$ have an Ahlfors-regular boundary.
\begin{lem}\label{lem_AF_boundary2D}
    We work in $\R^2$.
    Let $K$ be a $M$-quasiminimal set in $\Omega$, let $x_0 \in K$ and $r > 0$ such that $B(x_0,r) \subset \Omega$. For all connected component $V$ of $\Omega \setminus K$ such that $V \cap B(x_0,r/2) \ne \emptyset$, we have
    \begin{equation*}
        \HH^1(\partial V \cap B(x_0,r)) \geq r/2.
    \end{equation*}
\end{lem}
\begin{proof}
    Let $\rho \in (r/2,r)$.
    According to Proposition \ref{prop_isolated}, we cannot have $V \subset B(x_0,\rho)$. Thus there exists a continuous path $\gamma$ contained in $V$ which leaves $B(x_0,r/2)$ and arrives in $\Omega \setminus B(x_0,\rho)$. For all $t \in (r/2,\rho)$, the sphere $\partial B(x_0,t)$ must meet $\gamma$ so it must also meet $V$. Lemma \ref{lem_separation} shows that we can also find in $B(x_0,r/2) \setminus V$ another connected component, say $W$, of $\Omega \setminus K$. Reasoning like before, we see that for all $t \in (r/2,\rho)$, the sphere $\partial B(x_0,t)$ must meet $W$ and thus $\Omega \setminus V$. It follows that for all $t \in (r/2,\rho)$, the sphere $\partial B(x_0,t)$ meets $\partial V$. We then apply the co-area formula
    \begin{equation*}
        \HH^{1}(V \cap B(x_0,r)) \geq \int_{r/2}^\rho \HH^0(V \cap \partial B(0,t)) \dd{t} \geq \rho - r/2.
    \end{equation*}
    and let $\rho \to r$
\end{proof}

\begin{cor}[Local finiteness]\label{cor_localfiniteness}
    Let $K$ be a $M$-quasiminimal set in $\Omega$ and assume that $N = 2$. For all $x_0 \in K$ and $r > 0$ such that $B(x_0,r) \subset \Omega$, 
    \begin{equation*}
        \text{$B(x_0,r/2)$ meets at most $C$ components of $\Omega \setminus K$,}
    \end{equation*}
    where $C \geq 1$ depends only on $N$ and $M$.
\end{cor}
\begin{proof}
    Let $x_0 \in K$ and $r > 0$ be such that $B(x_0,r) \subset \Omega$. Let $(V_i)_i$ be the connected components of $B(x_0,r) \setminus \Omega$. For all $i$ such that $V_i$ meets $B(x_0,r/2)$, we have
    \begin{equation*}
        \HH^{1}(\partial V_i \cap B(x_0,r)) \geq C^{-1} r
    \end{equation*}
    and as $\Omega \cap \partial V_i \subset K$ and $\HH^{N-1}(K \setminus K^*) = 0$, we also have
    \begin{equation}\label{eq_bound0}
        \HH^{1}(K^* \cap \partial V_i \cap B(x_0,r)) \geq C^{-1} r.
    \end{equation}
    For all point $x \in K^*$, there exists exactly two indices $i \ne j$ such that $x \in \partial V_i \cap \partial V_j$ and therefore $\sum_i \mathbf{1}_{\partial V_i} = 2$ on $K^*$. It follows that
    \begin{equation}\label{eq_bound1}
        \sum_i \HH^{N-1}(K^* \cap \partial V_i \cap B(x_0,r)) \leq \HH^{N-1}(K \cap B(x_0,r)) \leq C r^{N-1}.
    \end{equation}
    Combining (\ref{eq_bound0}) and (\ref{eq_bound1}), we conclude that the number of indices $i$ such that $V_i$ meets $B(x_0,r/2)$ is bounded depending on $M$.
\end{proof}


The argument of Lemma \ref{lem_AF_boundary2D} is very specific to the dimension 2. In the next section, we present a proof of local finiteness which is suitable in $\R^3$.

\subsection{In the three-dimensional space}

\begin{thm}[Local finiteness]\label{thm_finiteness}
    Let $K$ be a $M$-quasiminimal set in $\Omega$ and assume that $N = 2$ or $3$.
    For all $x_0 \in K$, for all $r > 0$ such that $B(x_0,r) \subset \Omega$, 
    \begin{equation*}
        \text{$B(x_0,r/2)$ meets at most $C$ components of $\Omega \setminus K$,}
    \end{equation*}
    where $C \geq 1$ depends only on $M$.
\end{thm}
\begin{proof}
    For the moment, we let $N$ be any integer greater than or equal $2$. We shall only consider the cases $N=2$ or $3$ at the end of the proof. So let us fix $x_0 \in K$. For $r > 0$, we let $B_r$ denotes the open ball of center $x_0$ and radius $r$. We fix some radius $R > 0$ such that $B_R \subset \Omega$. We denote $(V_i)_{i \geq 1}$ the connected components of $\Omega \setminus K$. If the sequence is finite, we complete it by setting $V_i = \emptyset$ so that it is defined for all indices $i \geq 1$. For all indices $i,j$, we then set
    \begin{equation*}
        V_{ij} =
        \begin{cases}
            \partial V_i \cap \partial V_j &\ \text{if $i \ne j$}\\
            \emptyset                      &\ \text{if $i = j$.}
        \end{cases}
    \end{equation*}
    We assume that the components $(V_i)_{i \geq 1}$ are ordered in such a way that $$\abs{B_R \cap V_i} \geq \abs{B_R \cap V_{i+1}}$$
    and we assume that for some exponent $\alpha > 1$ and some constant $C_0 \geq 1$, we have
    \begin{equation}\label{eq_GH}
        \abs{B_R \cap V_k} \leq \frac{C_0 R^N}{k^{\alpha}} \quad \text{for all $k \geq 1$}.
    \end{equation}
    Let us justify that such an assumption always hold true, at least for $\alpha = N/(N-1)$. We start by recalling the following standard Lemma.
    \begin{lem}\label{lem_serie}
        Let $(c_k)_{k \geq 1}$ be a non-increasing sequence of non-negative real numbers such that $\sum_i c_i < +\infty$. Then for all $k \geq 1$,
        \begin{equation*}
            c_k \leq \frac{C}{k} \sum_{i \geq 1} c_i,
        \end{equation*}
        where $C \geq 1$ is a universal constant.
    \end{lem}
    According to the relative isoperimetric inequality (Remark \ref{rmk_isoperimetry}) and the Ahlfors-regularity of $K$ (Proposition \ref{prop_af}), we have
    \begin{equation*}
        \sum_{k \geq 1} \abs{B_R \cap V_k}^{(N-1)/N} \leq C \sum_{k \geq 1} \HH^{N-1}(B_R \cap \partial V_k) \leq C R^{N-1}.
    \end{equation*}
    Then, we apply Lemma \ref{lem_serie} and conclude that (\ref{eq_GH}) holds true for $\alpha = N/(N-1)$ and a constant $C_0$ which depends on $N$ and $M$.

    We are going to see that an assumption such as (\ref{eq_GH}) implies a better decay at a smaller scale, namely (up to reordering the sequence)
    \begin{equation*}
        \abs{B_{R/2} \cap V_k} \leq \frac{C R^N}{k^{f(\alpha)}} \quad \text{for all $k \geq 1$},
    \end{equation*}
    where $C \geq 1$ is a constant which depends on $C_0$, $N$, $M$, $\alpha$ and $f$ is the function
    \begin{equation*}
        f : t \mapsto \frac{N}{N-1}\left(t + t^{-1} - 1\right).
    \end{equation*}
\begin{figure}[ht]
\begin{center}
\includegraphics[height=140pt]{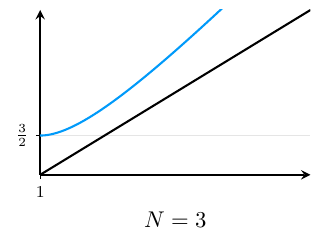}
\includegraphics[height=140pt]{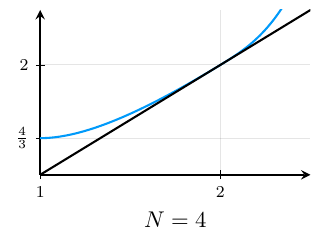}
\end{center}
\begin{center}
\includegraphics[height=140pt]{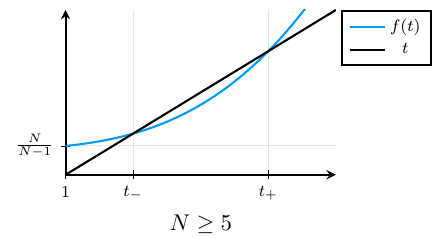}
\end{center}
\caption{The function $f$ and its fixed points in Lemma~\ref{lem_serie}.}
\end{figure}
When $N = 2$ or $3$, we have $f(t) > t$ for all $t > 1$, allowing the procedure to be iterated indefinitely. Besides, the successive iterates $f(\alpha), f^{(2)}(\alpha), ...$ of $\alpha := N/(N-1)$ go to $\infty$ and we will see that if the volumes $(\abs{B_r \cap V_k})_k$ decay too fast in a ball $B_r$, then almost-all components have zero volume in $B_{r/2}$. When $N = 4$, the function $f$ has a fixed point at $t_{-} = t_{+} := 2$ and we have $f(t) > t$ for $t \in (1,2)\cup (2,+\infty)$. If $N \geq 5$ then $f$ has two fixed points $$t_{\pm} := (N \pm \sqrt{N(N-4)})/2$$ and we have $f(t)>t$ for $t\in (1;t_-)\cup (t_+ ;\infty)$ and $f(t)<t$ for $t \in (t_-;t_+)$. In these two last cases, the successive iterates of $\alpha := N/(N-1)$ converge to the fixed point $t_-$ but this does not provide a contradiction. Reasonning as in the $N=2,3$ case, our argument would nevertheless imply that for all $s < t_-$, we have $\abs{V_k \cap B_r} \lesssim k^{-s}$ in a ball $B_r$ of sufficiently small radius (depending on $s$) and that if $\abs{V_k \cap B_R} \lesssim k^{-s}$ for some $s > t_+$, then almost-all components have zero volume in a smaller ball $B_r$ (with $r$ depending on $s$).


    In what follows, the letter $C$ stands for a generic constant $\geq 1$ which depends on $N$, $M$, $C_0$ and $\alpha$.
    Let $r \in (R/2,R)$. By the relative isoperimetric inequality (Remark \ref{rmk_isoperimetry}), and the properties off $(V_i)_i$ (Proposition \ref{prop_generic_points}), we have for all $j \geq 1$,
    \begin{equation*}
        \abs{B_r \cap V_j}^{(N-1)/N} \leq \HH^{N-1}(B_r \cap \partial V_j) \leq \sum_i \HH^{N-1}(B_r \cap V_{ij})
    \end{equation*}
    and then for all $k \geq 1$,
    \begin{equation*}
        \sum_{j \geq k} \abs{B_r \cap V_j}^{(N-1)/N} \leq C \sum_{j \geq k} \sum_{i} \HH^{N-1}(B_r \cap V_{ij}).
    \end{equation*}
    In order to estimate the right-hand side, we are going to use the quasiminimality property by merging components into others (see Remark \ref{rmk_merging} and in particular (\ref{eq_merging})). For $i \ne j$, we merge $V_j$ into $V_i$ to obtain
    \begin{equation}\label{eq_GQ1}
        \HH^{N-1}(B_r \cap V_{ij}) \leq \HH^{N-1}(V_j \cap \partial B_r)
    \end{equation}
    and for all $k \geq 1$, we merge all the components $(V_i)_{i > k}$ into $V_k$ to obtain
    \begin{equation}\label{eq_GQ2}
        \sum_{i,j \geq k} \HH^{N-1}(B_r \cap V_{ij}) \leq C \sum_{j > k} \HH^{N-1}(V_j \cap \partial B_r).
    \end{equation}
    We decompose
    \begin{equation*}
        \sum_{j \geq k} \sum_{i} \HH^{N-1}(B_r \cap V_{ij}) \leq \sum_{j \geq k} \sum_{i=1}^j \HH^{N-1}(B_r \cap V_{ij}) + \sum_{j \geq k} \sum_{i \geq j} \HH^{N-1}(B_r \cap V_{ij})
    \end{equation*}
    and (\ref{eq_GQ2}) already allows to estimate the second term at the right-hand side
    \begin{equation*}
        \sum_{j \geq k} \sum_{i \geq j} \HH^{N-1}(B_r \cap V_{ij}) \leq C \sum_{j > k} \HH^{N-1}(V_j \cap \partial B_r).
    \end{equation*}
    To estimate the first term, we consider some sequence non-decreasing $(N_k)_{k \geq 1}$ of positive integers such that $N_k \leq k$ but which is left unknown for the moment. We use Fubini, (\ref{eq_GQ1}), (\ref{eq_GQ2}) to estimate
    \begin{align*}
        \sum_{j \geq k} \sum_{i=1}^j \HH^{N-1}(B_r \cap V_{ij}) &\leq \sum_{j \geq k} \sum_{i=1}^{N_j} \HH^{N-1}(B_r \cap V_{ij}) + \sum_{j \geq k} \sum_{i=N_j}^j \HH^{N-1}(B_r \cap V_{ij})\\
                                                                &\leq \sum_{j \geq k} \sum_{i=1}^{N_j} \HH^{N-1}(B_r \cap V_{ij}) + \sum_{i \geq N_k} \sum_{j \geq i} \HH^{N-1}(B_r \cap V_{ij})\\
                                                                &\leq C\sum_{j \geq k} N_j \HH^{N-1}(V_j \cap \partial B_r) + C\sum_{i \geq N_k} \HH^{N-1}(V_i \cap \partial B_r).
    \end{align*}
    In conclusion, we obtained
    \begin{multline*}
        \sum_{j \geq k} \abs{B_r \cap V_j}^{(N-1)/N} \leq C\sum_{j \geq k} N_j \HH^{N-1}(V_j \cap \partial B_r) \\+ C\sum_{i \geq N_k} \HH^{N-1}(V_i \cap \partial B_r) + \sum_{j > k} \HH^{N-1}(V_j \cap \partial B_r).
    \end{multline*}
    Integrating over $r \in [R/2,R]$ and in view of the co-area formula
    \begin{equation*}
        \int_{R/2}^R \HH^{N-1}(V_i \cap \partial B_r) \dd{r} \leq C \abs{B_R \cap V_i} \quad \quad \text{for all $i$,}
    \end{equation*}
    we deduce that
    \begin{equation*}
        \sum_{j \geq k} \abs{B_{R/2} \cap V_j}^{(N-1)/N} \leq \frac{C}{R} \sum_{j \geq k} N_j \abs{B_R \cap V_j} + \frac{C}{R} \sum_{i \geq N_k} \abs{B_R \cap V_i} + \frac{C}{R} \sum_{j > k} \abs{B_R \cap V_j}.
    \end{equation*}
    We recall that by assumption (\ref{eq_GH}), we have $\abs{B_R \cap V_k} \leq C_0 R^N/k^{\alpha}$ for all $k \geq 1$, so
    \begin{equation}\label{eq_GCCL0}
        \sum_{j \geq k} \abs{B_{R/2} \cap V_j}^{(N-1)/N} \leq C R^{N-1} \sum_{j \geq k} \frac{N_j}{j^{\alpha}} + C R^{N-1} N_k^{1-\alpha} + \frac{C R^{N-1}}{k^{\alpha-1}}.
    \end{equation}
    The optimal choice for $(N_k)_{k \geq 1}$ is to take $N_k$ as the integer part of $k^{(\alpha-1)/\alpha}$. In particular, notice that we always have
    \begin{equation*}
        C^{-1} k^{(\alpha-1)/\alpha} \leq N_k \leq k^{(\alpha-1)/\alpha}
    \end{equation*}
    so (\ref{eq_GCCL0}) yields
    \begin{equation}\label{eq_GCCL}
        \sum_{j \geq k} \abs{B_{R/2} \cap V_j}^{(N-1)/N} \leq \frac{C R^{N-1}}{k^{(1-\alpha)^2/\alpha}}.
    \end{equation}
    The sequence $(\abs{B_{R/2} \cap V_k})_{k \geq 1}$ may not be non-increasing but if we re-order the sequence in a non-increasing way, the conclusion (\ref{eq_GCCL}) remains true because we have only made the left-hand side of (\ref{eq_GCCL}) smaller. We can then apply Lemma \ref{lem_serie2} below (the proof is the same as in Lemma \ref{lem_serie2} and is omitted) to deduce that for all $k \geq 1$,
    \begin{equation*}
        \abs{B_{R/2} \cap V_k} \leq \frac{C R^N}{k^{f(\alpha)}}, \quad \text{where} \quad f(t) = \frac{N}{N-1}\left(t + t^{-1} - 1\right).
    \end{equation*}
    
    \begin{lem}\label{lem_serie2}
        Let $(c_k)_k$ be a non-increasing sequence such that for all $k$, $c_k \geq 0$ and
        \begin{equation*}
            \sum_{i \geq k} c_i \leq \frac{C_0}{k^\gamma},
        \end{equation*}
        for some constants $C_0 \geq 1$ and $\gamma \geq 0$. Then for all $k \geq 1$,
        \begin{equation*}
            c_k \leq \frac{C}{k^{\gamma + 1}},
        \end{equation*}
        where $C \geq 1$ depends only on $C_0$ and $\gamma$.
    \end{lem}

    In the case $N=2$ or $3$, we can repeat the procedure indefinitely and obtain that the volumes $\abs{B_{2^{-p} R} \cap V_k}$ decay faster than $1/k^5$ after a finite number $p$ of iterations (which depends on $M$). In view of Lemma \ref{lem_control_decay} below, this shows that only a controlled number of $(V_i)_i$ can meet $B_{2^{-p-1} R}$. The statement of Theorem~\ref{thm_finiteness} follows by a covering argument.
\end{proof}


    \begin{lem}\label{lem_control_decay}
        Let $K$ be a $M$-quasiminimal set in $\Omega$. Let $(V_i)_{i \geq 1}$ denote the components of $\Omega \setminus K$, the sequence being finite or infinite. For all $x_0 \in K$ and $r > 0$ such that $B(x_0,r) \subset \Omega$, and for all $k \geq 1$,
        \begin{equation*}
            \sum_{i > k} \abs{B(x_0,r/2) \cap V_i} > 0 \ \implies \ \sum_{i > k} \abs{B(x_0,r) \cap V_i} \geq \frac{C^{-1} r^N}{k^N},
        \end{equation*}
        where $C \geq 1$ depends on $N$ and $M$.
    \end{lem}
    \begin{proof}
        If the sequence $(V_i)_{i \geq 1}$ is finite, we complete it by setting $V_i = \emptyset$ so that it is defined for all indices $i \geq 1$.
        Let $x_0 \in K$ and let $r > 0$ be such that $B(x_0,r) \subset \Omega$ and assume that for some $k \geq 1$,
        \begin{equation*}
            \sum_{i > k} \abs{B(x_0,r/2) \cap V_i} > 0.
        \end{equation*}
        For $s \in (0,r]$, we let $B_s$ denote $B(x_0,s)$ and for $i \ne j$, we set
        \begin{equation*}
            V_{ij} =
            \begin{cases}
                \partial V_i \cap \partial V_j &\ \text{if $i \ne j$}\\
                \emptyset                      &\ \text{if $i = j$.}
            \end{cases}
        \end{equation*}
        For $s \in (0,r)$ and $i = 1,\ldots,k$, we merge $(V_j)_{j > k}$ into $V_i$ within $B_s$ to obtain
        \begin{equation*}
            \sum_{j > k} \HH^{N-1}(B_s \cap V_{ij}) \leq C \sum_{j > k} \HH^{N-1}(\partial B_s \cap V_j).
        \end{equation*}
        By the relative isoperimetric inequality applied to $\left(\bigcup_{j > k} B_s \cap V_j\right)$ in $B_s$, we have
        \begin{align*}
            \left(\sum_{j > k} \abs{B_s \cap V_j}\right)^{(N-1)/N}  &\leq \sum_{i=1}^k \sum_{j > k} \HH^{N-1}(B_s \cap V_{ij})\\
                                                                    &\leq C k \sum_{j > k} \HH^{N-1}(\partial B_s \cap V_j).
        \end{align*}
        We set $f(s) := \sum_{j > k} \abs{B_s \cap V_j}$. The function $f$ is absolutely continuous on $(0,r)$, with
        \begin{equation*}
            f'(s) = \sum_{j > k} \HH^{N-1}(\partial B_s \cap V_j) \quad \text{for a.e. $s \in (0,r)$}.
        \end{equation*}
        The previous inequality tells us that for a.e. $s \in (r/2,r)$, we have
        \begin{equation*}
            f(s)^{(N-1)/N} \leq C k f'(s). 
        \end{equation*}
        As $f(r/2) > 0$ and $f(s) \in [f(r/2),f(r)]$ for all $s \in [r/2,r]$, we deduce that
        \begin{equation*}
            f(r)^{1/N} - f(r/2)^{1/N} \geq \frac{C^{-1} r}{k},
        \end{equation*}
        whence
        \begin{equation*}
            f(r) \geq \frac{C^{-1} r^N}{k^N}
        \end{equation*}
        for some bigger constant $C$.
    \end{proof}

\subsection{Applications}

As an application of the local finiteness, we are going to see that for all component $V$ of $\Omega \setminus K$, the boundary $\partial V$ locally satisfies Ahlfors-regularity and condition B. The statements are limited to $\R^2$ and $\R^3$ only because they rely on Theorem~\ref{thm_finiteness}.

\begin{prop}[Ahlfors-regularity for the boundaries]\label{propV_AF}
    Let $K$ be a $M$-quasiminimal set in $\Omega$ and assume that $N = 2$ or $3$.
    For all $x_0 \in K$, for all $r > 0$ such that $B(x_0,r) \subset \Omega$, for all connected components $V$ of $\Omega \setminus K$ such that $V \cap B(x_0,r/2) \ne \emptyset$, we have
    \begin{equation*}
        \abs{\partial V \cap B(x_0,r)} \geq C^{-1} r^N,
    \end{equation*}
    where $C \geq 1$ depends on $M$.
\end{prop}

\begin{prop}[Condition B for the boundaries]\label{propV_conditionB}
    Let $K$ be a $M$-quasiminimal set in $\Omega$ and assume that $N = 2$ or $3$.
    Let $x_0 \in K$, $r > 0$ be such that $B(x_0,r) \subset \Omega$. If $V$ is a connected component of $\Omega \setminus K$ such that $V \cap B(x_0,r/2) \ne \emptyset$, then 
    \begin{equation*}
        \text{both $B(x,r) \cap V$ and $B(x,r) \setminus V$ contain a ball of radius $\geq C^{-1} r$}, 
    \end{equation*}
    where $C \geq 1$ depends on $M$.
\end{prop}

In order to show Proposition \ref{propV_AF} and Proposition \ref{propV_conditionB}, we start by proving a lower bound on the volume of the components of $\Omega \setminus K$.
\begin{lem}\label{lemV_volume}
    Let $K$ be a $M$-quasiminimal set in $\Omega$ and assume that $N = 2$ or $3$.
    Let $x_0 \in K$, $r > 0$ be such that $B(x_0,r) \subset \Omega$. If $V$ is a connected component of $\Omega \setminus K$ such that $V \cap B(x_0,r/2) \ne \emptyset$, then
    \begin{equation*}
        \abs{V \cap B(x_0,r)} \geq C^{-1} r^{N-1},
    \end{equation*}
    where $C \geq 1$ depends on $M$.
\end{lem}
\begin{proof}
The proof is in the same spirit as Proposition \ref{prop_af}.
Let $x_0 \in K$ and $r > 0$ be such that $B(x_0,r) \subset \Omega$ and let $V$ be a connected component of $\Omega \setminus K$.
Let $(V_i)_{i \geq 1}$ denote the connected components of $\Omega \setminus K$ with $V_1 = V$.

\emph{Step 1.} We are going to prove that there exists some small $\varepsilon > 0$ (which depends only on $M$) such that
\begin{equation}\label{eq_BBAF1}
    \abs{V \cap B(x_0,r)} \leq \varepsilon r^N \implies \lim_{r \to 0} r^{-N} \abs{V \cap B(x_0,r)} = 0.
\end{equation}
Let us assume that $\abs{V \cap B(x_0,r)} \leq \varepsilon r^N$ for some $\varepsilon > 0$ to be chosen later.
According to the co-area formula
\begin{equation*}
    \abs{V \cap B(x_0,r)} = \int_{0}^r \HH^{N-1}(V \cap \partial B(x_0,t)) \dd{t}
\end{equation*}
we can find a radius $t \in (r/2,r)$ such that
\begin{equation}\label{eq_selection_t}
    \HH^{N-1}(V \cap \partial B(x_0,t)) \leq C \varepsilon r^{N-1}.
\end{equation}
By Theorem~\ref{thm_finiteness}, there exists at most $m$ components of the family $(V_i)_{i \geq 1}$ such that $$V_i \cap B(x_0,r/2) \ne \emptyset,$$ where $m \geq 2$ depends on $M$. We re-order the indices so that $V_i \cap B(x_0,r/2) = \emptyset$ for $i > m$.
By properties of $(V_i)_{i \geq 1}$ stated in Proposition \ref{prop_generic_points}, we have
\begin{equation*}
    \HH^{N-1}(B(x_0,r/2) \cap \partial V) = \sum_{i = 2}^m \HH^{N-1}(B(x_0,r/2) \cap \partial V \cap \partial V_i)
\end{equation*}
so there exists an index $i = 2,\ldots, m$ such that
\begin{equation}\label{eq_selection_i}
    \HH^{N-1}(B(x_0,r/2) \cap \partial V \cap \partial V_i) \geq C^{-1} \HH^{N-1}(B(x_0,r/2) \cap \partial V).
\end{equation}
By merging $V$ into $V_i$ within $B(x_0,t)$, see Remark \ref{rmk_merging}, we know that
\begin{equation*}
    \HH^{N-1}(B(x_0,t) \cap \partial V \cap \partial V_i) \leq \HH^{N-1}(V \cap \partial B(x_0,t))
\end{equation*}
and thus by (\ref{eq_selection_t}), (\ref{eq_selection_i})
\begin{equation*}
    \HH^{N-1}(B(x_0,r/2) \cap \partial V) \leq C \varepsilon r^{N-1}.
\end{equation*}
Then, we apply the relative isoperimetric inequality to $V$ in $B(x_0,r/2)$, see Remark \ref{rmk_isoperimetry},
\begin{equation*}
    \abs{V \cap B(x_0,r/2)} \leq C \HH^{N-1}(B(x_0,r/2) \cap \partial V)^{N/(N-1)} \leq C \varepsilon^{N/(N-1)} r^N.
\end{equation*}
Choosing $\varepsilon$ small enough, we see that
\begin{equation*}
    \abs{V \cap B(x_0,r)} \leq \varepsilon r^N \quad \implies \quad \abs{V \cap B(x_0,r/2)} \leq \varepsilon (r/2)^N
\end{equation*}
and the argument can be iterated to concludes Step 1.

\emph{Step 2.} The parameter $\varepsilon$ being fixed as in Step 1, we observe that if $\abs{V \cap B(x_0,r)} \leq \varepsilon (r/2)^N$, then for all $x \in K \cap B(x_0,r/2)$, we have $\abs{V \cap B(x,r/2)} \leq \varepsilon (r/2)^N$ so
\begin{equation}\label{eq_xV2}
\lim_{r \to 0} r^{-N} \abs{V \cap B(x,r)} = 0. 
\end{equation}
It follows that $V \cap B(x_0,r/2) = \emptyset$ by Lebesgue's density theorem.
This proves the Proposition by contraposition. 
\end{proof}

\begin{proof}[Proof of Proposition \ref{propV_AF}]
    This is a direct application of Lemma \ref{lemV_volume}, together with the relative isoperimetric inequality in $B(x_0,r)$, see Remark \ref{rmk_isoperimetry}.
\end{proof}

\begin{proof}[Proof of Proposition \ref{propV_conditionB}]
    Let $V$ be a connected component of $\Omega \setminus K$.
    Let $x \in V$ and $r > 0$ be such that $B(x,r) \subset \Omega$. If $B(x,r/4) \subset V$, then the proof is over. Otherwise, there exists a point $x_0 \in K \cap B(x,r/4)$ and as $V \cap B(x_0,r/4) \ne \emptyset$, it follows from Proposition \ref{propV_AF} that
    \begin{equation*}
        \HH^{N-1}(\partial V \cap B(x_0,r/2)) \geq C^{-1} r^{N-1}.
    \end{equation*}
    We also know by uniform rectifiability of $K$ that for all $\varepsilon > 0$, there exists a constant $C_0 = C_0(\varepsilon) \geq 1$ (which depends on $N$, $M$, and $\varepsilon$) such that
    \begin{equation*}
        \int_{K \cap B(x_0,r/2)} \int_{0}^{r/2} \mathbf{1}_{\beta(y,t) \geq \varepsilon}(y,t) \frac{\dd{t}}{t} \dd{\HH^{N-1}(y)} \leq C_0 r^{N-1}.
    \end{equation*}
    As $\HH^{N-1}(\partial V \cap B(x_0,r/2)) \geq C^{-1} r^{N-1}$, we prove by contradiction (as in Corollary \ref{cor_fl}) that their exists a point $y_0 \in \partial V \cap B(x_0,r/2)$ and a radius $t \in (C_1^{-1}r,r/2)$ such that
    \begin{equation*}
        \beta(y_0,t) \leq \varepsilon,
    \end{equation*}
    where $C_1 = C_1(\varepsilon) \geq 1$ depends on $N$, $M$ and $\varepsilon)$.
    We choose $\varepsilon$ small enough (depending on $N$, $M$ and $m$) so that Proposition \ref{prop_plane_separation} yields that that $B(y_0,t/2) \setminus K$ has exactly two components and as, $y_0 \in \partial V$, one of them is contained in $V$. The other cannot be contained in $V$ because of Lemma \ref{lem_separation}.
    Finally we observe that since $K$ has a a flatness $\leq 2 \varepsilon$ in $B(y_0,t/2)$, with say $\varepsilon \leq 1/100$, one can build a ball of radius $t/10$ in each component of $B(y_0,t/2) \setminus K$.
\end{proof}

\section{Quasiminimal sets and John domains}

\subsection{Components of \texorpdfstring{$\Omega \setminus K$}{Omega minus K} are local John domains}\label{section_john}

A John domain is a bounded open set $V \subset \R^N$ where any point $z \in V$ can be connected to a center $z_0 \in V$ via a path that gets away fast enough from the boundary.
Examples of John domains are Lipschitz and snowflake domains. The main counter-example are the domains with outward cusps. 

\begin{defi}[John Domain]\label{defi_john}
    Let $V$ be a bounded open set and $z_0 \in V$. We say that $V$ is a John domain of center $z_0$ if there exists a constant $C \geq 1$ such that for all $z \in V$, there exists a $C$-Lipschitz path $\gamma :[0,\ell] \to V$, where $\ell = \abs{z - z_0}$, such that $\gamma(0) = z$, $\gamma(\ell) = z_0$ and
    \begin{equation*}
        \mathrm{dist}(\gamma(t),\R^N \setminus V) \geq C^{-1} t \quad \text{for all $t \in [0,\ell]$}.
    \end{equation*}
\end{defi}
The property also extends to points $z \in \partial V$ by an application of Arzela-Ascoli theorem.
Our goal is to prove that the components of $\Omega \setminus K$ are local John domains in the following sense.
\begin{thm}[Existence of escape paths]\label{thm_john}
    Let $K$ be a $M$-quasiminimal set in $\Omega$ and assume that $N = 2$ or $3$. Let $V$ be a component of $\Omega \setminus K$. Then for all $x \in V$, there exists a $1$-Lipschitz path $\gamma : [0,\ell] \to V$, where $\ell := \mathrm{dist}(x,\partial \Omega)$, such that $\gamma(0) = x$ and
    \begin{equation*}
        \mathrm{dist}(\gamma(t),K) \geq C^{-1} t \quad \text{for all $t \in [0,\ell]$.}
    \end{equation*}
\end{thm}

This property states that any point $z \in V$ has an escape path that gets away fast enough from $K$. This provides a good nontangential access region to $z$. In contrast to Definition \ref{defi_john} however, it does not request a single center where all escape paths ends. The existence of escape paths has the following easy consequence. 

\begin{cor}[Local truncations are John domains]
    There exists a constant $C_0 \geq 2$ which depends only on $M$ such that the following holds true.
    Let $K$ be a $M$-quasiminimal set in $\Omega$ and assume that $N = 2$ or $3$.
    For all $x_0 \in K$ and $r > 0$ with $B(x_0,C_0 r) \subset \Omega$ and for all connected component $V$ of $\Omega \setminus K$, the set
    \begin{equation}\label{eq_john_truncate}
        [V \cap B(x_0,2 r)] \cup [B(x_0,2 r) \setminus \overline{B}(x_0,r)]
    \end{equation}
    is a John domain with constant $C_0$.
\end{cor}
This is a quantitative strengthening of the fact that $V$ cannot stay within $B(x_0,r)$, seen in Proposition \ref{prop_isolated}. An alternative way of truncating of $V$ could be to take the intersection $V \cap B(x_0,r)$ but this does not produce a John domain in general as $V \cap B(x_0,r)$ can be disconnected or have an outward cusp near $\partial B(x_0,r) \cap \partial V$.

John domains are closely related to domains of isometry \cite[Definition 5.1]{DS98}, whose definition we recall below.
\begin{defi}
    An open set $V \subset \R^N$ is a domain of isoperimetry if there exists a constant $C \geq 1$ such that for all open subset $Z \subset V$, we have
    \begin{equation*}
        \min(\abs{Z},\abs{V \setminus Z})^{(N-1)/N} \leq C \HH^{N-1}(\partial Z \cap V).
    \end{equation*}
\end{defi}
Bojarski \cite{Bojarski} proved that John domains are domains of isoperimetry but the converse is not true in general, see for instance the generalized Eiffel tower in \cite[Remark 6.2]{DS98}.
There are however some partial converse results due to Buckley--Koskela \cite[Theorem 1.1]{Buckley} and David--Semmes \cite[Theorem 6.1]{DS98}. We are going to establish first in Proposition \ref{prop_isometry} that the components of $\Omega \setminus K$ satisfy a local relative isoperimetric inequality. We will then deduce Theorem~\ref{thm_john} using the method of \cite[Theorem 6.1]{DS98}.

\begin{prop}\label{prop_isometry}
    Let $K$ be a $M$-quasiminimal set in $\Omega$ and assume that $N = 2$ or $3$.
    Then, for all $x_0 \in K$ and $r > 0$ such that $B(x_0,2r) \subset \Omega$, for all connected component $V$ of $\Omega \setminus K$, and for all open subset $Z \subset V \cap B(x_0,r)$, one has
    \begin{equation*}
        \abs{Z}^{(N-1)/N} \leq C \HH^{N-1}(\partial Z \cap V),
    \end{equation*}
    where $C \geq 1$ depends only on $M$.
\end{prop}
\begin{proof}
    Let $x_0 \in K$ and $r > 0$ be such that $B(x_0,2r) \subset \Omega$.
    Let $V$ be a connected component of $\Omega \setminus K$ and let $Z$ be an open subset of $V \cap B(x_0,r)$.
    We let $(V_i)_{i \geq 1}$ denote the components of $\Omega \setminus K$ with $V_1 = V$. Since $B(x_0,r)$ meets a most $m$ components (where $m \geq 2$ depends on $M$), we can order the family so that for all $i > m$, $V_i \cap B(x_0,r) = \emptyset$. According the properties of $(V_i)_{i \geq 1}$ (Proposition \ref{prop_generic_points}), we have
    \begin{equation*}
        \HH^{N-1}(\partial^* Z \cap \partial V) = \sum_{i=2}^m \HH^{N-1}(\partial Z \cap \partial V \cap \partial V_i)
    \end{equation*}
    so we can find an index $i = 1,\ldots,m$ such that
    \begin{equation}\label{eq_iso1}
        \HH^{N-1}(\partial^* Z \cap \partial V) \leq C \HH^{N-1}(\partial Z \cap \partial V \cap \partial V_i),
    \end{equation}
    where $C \geq 1$ depends on $m$.
    Now we want to prove that
    \begin{equation}\label{eq_iso2}
        \HH^{N-1}(\partial^* Z \cap \partial V \cap \partial V_i) \leq M \HH^{N-1}(\partial Z \cap V),
    \end{equation}
    by merging $Z$ in $V_i$, that is, building a competitor by removing $K \cap B(x_0,r) \cap \partial Z \cap \partial V$ and adding $\partial Z \cap V$.
    From (\ref{eq_iso1}) and (\ref{eq_iso2}), we would have
    \begin{equation*}
        \HH^{N-1}(\partial^* Z \cap \partial V) \leq C \HH^{N-1}(\partial Z \cap V).
    \end{equation*}
    and finally by applying the isoperimetric inequality in $\R^N$, we would conclude
    \begin{align*}
        \abs{Z}^{(N-1)/N} &\leq C \HH^{N-1}(\partial^* Z)\\
                &\leq C \HH^{N-1}(\partial^* Z \cap \partial V) + C \HH^{N-1}(\partial^* Z \cap V)\\
                &\leq C \HH^{N-1}(\partial Z \cap V).
    \end{align*}

As in Lemma \ref{lem_infiltration}, the proof of (\ref{eq_iso2}) by merging poses a technical difficulty because competitors must be closed subset of $\Omega$. We are going to build the competitor by a covering argument. We first justify that for $\HH^{N-1}$-a.e. $x \in B(x_0,r) \cap \partial^* Z \cap \partial V \cap \partial V_i$, one can find arbitrary small balls centered on $x$ with good properties. By standard theorems on reduced boundaries \cite[Theorem 3.61 and Example 3.68]{AFP}, we know that for $\HH^{N-1}$-a.e. $x \in \partial^* Z$, there exists a unit vector $\nu_0 \in \R^N$ such that
    \begin{equation}\label{eq_ZDH}
        \lim_{t \to 0} t^{-N} \abs{B(x,t) \cap Z \Delta H_0^+} = 0,
    \end{equation}
    where $H_0^+ = \set{y \in \R^N | (y - x) \cdot \nu_0 \geq 0}$.
    Similarly, since $\HH^{N-1}(\Omega \cap \partial V \setminus \partial^* V) = 0$ (Proposition \ref{prop_generic_points}), then for $\HH^{N-1}$-a.e. $x \in \Omega \cap \partial V$, there exists a unit vector $\nu_1 \in \R^N$ such that
    \begin{equation}\label{eq_VDH}
        \lim_{t \to 0} t^{-N} \abs{B(x,r) \cap V \Delta H_1^+} = 0,
    \end{equation}
    where $H_1^+ = \set{y \in \R^N | (y - x) \cdot \nu_1 \geq 0}$.
    At a point $x$ where both (\ref{eq_ZDH}) and (\ref{eq_VDH}) hold, we see that we must actually have $\nu = \nu_1$ and thus
    \begin{equation*}
        \lim_{t \to 0} t^{-N} \abs{B(x,t) \cap V \setminus Z} = 0.
    \end{equation*}
    We observe finally that for all small $t > 0$, the co-area formula
    \begin{equation*}
        \abs{B(x,t) \cap V \setminus Z} = \int_0^t \HH^{N-1}(\partial B(s,t) \cap V \setminus Z) \dd{s}
    \end{equation*}
    allows to find a radius $s \in (t/2,t)$ such that
    \begin{equation*}
        \HH^{N-1}(\partial B(x,s) \cap V \setminus Z) \leq C t^{-1} \abs{B(x,t) \cap V \setminus Z}.
    \end{equation*}
    We can also assume that $s$ satisfies $\HH^{N-1}(K \cap \partial B(x,s)) = 0$ since this holds for a.e. $s \in (0,t)$.

    Now, we fix $\varepsilon > 0$ and we apply the Vitali covering theorem \cite[Theorem 2.8]{mattila} to find a family of balls $(B(y_k,t_k))_{k}$, where $y_k \in B(x_0,r) \cap \partial^* Z \cap \partial V \cap \partial V_i$ and $t_k > 0$, such that the closed balls $(\overline{B}(y_k,t_k))_k$ are disjoint,
    \begin{equation}\label{eq_coveringZ}
        \HH^{N-1}\Bigl(B(x_0,r) \cap \partial^* Z \cap \partial V \cap \partial V_i \setminus \bigcup_k \overline{B}(y_k,t_k)\Bigr) = 0
    \end{equation}
    and for all $k$,
    \begin{enumerate}[label = (\roman*)]
    \item $\overline{B}(y_k,t_k) \subset B(x_0,r)$;
    \item $\HH^{N-1}(K \cap \partial B(y_k,t_k)) = 0$;
    \item $\overline{B}(y_k,t_k)$ meets exactly two components of $\Omega \setminus K$;
    \item $\HH^{N-1}(\partial B(y_k,t_k) \cap V \setminus Z) \leq \varepsilon t_k^{N-1}$.
\end{enumerate}
Note that the balls are centred in $K$ since $\partial V \subset K$. As $\HH^{N-1}(B(x_0,r) \cap \partial^* Z \cap \partial V_i) < +\infty$, we can assume that the family is finite but at the price of replacing (\ref{eq_coveringZ}) by 
    \begin{equation}\label{eq_coveringZ2}
        \HH^{N-1}\Bigl(B(x_0,r) \cap \partial^* Z \cap \partial V \cap \partial V_i \setminus \bigcup_k \overline{B}(y_k,t_k)\Bigr) \leq \varepsilon. 
    \end{equation}
    We define
    \begin{equation*}
        F := \Bigl(K \setminus \bigcup_{k} B(y_k,t_k)\Bigr) \cup \Bigl(\bigcup_{k} \partial B(y_k,t_k) \cap V \setminus Z\Bigr) \cup \Bigl(\partial Z \cap V\Bigr).
    \end{equation*}
    We then justify that $F$ is a topological competitor of $K$ in $\overline{B}(x_0,r)$.
    Since $Z$ is open and $\Omega \cap \partial V \subset K$, we see that $K \cup (\partial Z \cap V)$ and $K \cup (\partial B(y_k,t_k) \cap V \setminus Z)$ are relatively closed subsets of $\Omega$.
    Therefore, $F$ is relatively closed in $\Omega$ as a finite union of relatively closed sets.
    Now, let $x$ and $y$ be two points in $\Omega \setminus (\overline{B}(x_0,r) \cup K)$ which are connected by a path in $\Omega \setminus F$. We want to show that $x$ and $y$ lie in the same component of $\Omega \setminus K$.
    Assuming by contradiction that the two points $x$ and $y$ are not in the same component of $\Omega \setminus K$, then the path should meet $K$ and this can only happen within the interior of a ball $B(y_k,t_k)$. We consider the portion $\gamma_x$ of the path which starts from $x$ and meet a ball $\partial B(y_k,t_k)$ for the first time at some point $z_x$. Note that $\gamma_x$ connects $x$ to $z_x$ in $\Omega \setminus (K \cup F)$. In view of the definition of $F$, we see that $z$ is either in $V_i$ or in $Z$. We can observe similarly that $y$ is connected by a path $\gamma_y$ in $\Omega \setminus (K \cup F)$ to a point $z_y$ which is either in $V_i$ or in $V$. Since $x$ and $y$ cannot be both connected to a point of $V_i$ in $\Omega \setminus K$, either $z_x$ or $z_y$ must be in $Z$, let's say $z_x$. We finally reach a contradiction because if $\gamma_x$ connects $x$ to $z_x \in Z$ in $\Omega \setminus K$, then it must cross $\partial Z \setminus K \subset \partial Z \cap V \subset F$.
    We conclude that $F$ is a topological competitor of $K$ in $\overline{B}(x_0,r)$.

    By applying the quasiminimality property (\ref{eq_quasi}) in $\overline{B}(x_0,r)$,
    \begin{equation*}
        \HH^{N-1}(K \setminus F) \leq M \HH^{N-1}(F \setminus K),
    \end{equation*}
    we get
    \begin{align*}
        \HH^{N-1}(K \cap \bigcup_{k} B(y_k,t_k)) &\leq M \HH^{N-1}(\partial Z \cap V) + M \sum_{k} \HH^{N-1}(\partial B(y_k,t_k) \cap V \setminus Z)\\
                                                  &\leq M \HH^{N-1}(\partial Z \cap V) + C \varepsilon \sum_{k} t_k^{N-1}.
    \end{align*}
    By (\ref{eq_coveringZ2}) and the properties of the family of balls, we have
    \begin{align*}
        \HH^{N-1}(B(x_0,r) \cap \partial^* Z \cap \partial V \cap \partial V_i) &\leq \HH^{N-1}(K \cap \bigcup_k \overline{B}(y_k,t_k)) + \varepsilon\\
                                                                                &\leq \HH^{N-1}(K \cap \bigcup_k B(y_k,t_k)) + \varepsilon.
    \end{align*}
    Since the balls $(B(y_k,t_k))_k$ are disjoint and contained in $B(x_0,r)$ and since $K$ is Ahlfors-regular, we can also bound from above
    \begin{equation*}
        C \varepsilon \sum_{k} t_k^{N-1} \leq C \varepsilon \HH^{N-1}(B(x_0,r) \cap K);
    \end{equation*}
    in particular this term goes to $0$ when $\varepsilon \to 0$. We conclude that
    \begin{equation*}
        \HH^{N-1}(B(x_0,r) \cap \partial^* Z \cap \partial V \cap \partial V_i) \leq C \HH^{N-1}(\partial Z \cap V),
    \end{equation*}
    where $C \geq 1$ depends on $N$ and $M$.
    This proves (\ref{eq_iso2}) and ends the proof.

\end{proof}

As the components of $\Omega \setminus K$ have a boundary which is locally Ahlfors-regular (Proposition \ref{propV_AF}), which satisfies locally the condition B (Proposition \ref{propV_conditionB}) and which supports a local relative isoperimetric inequality (Proposition \ref{prop_isometry}), we can deduce the following Lemma. This is a minor variation of the proof of \cite[Theorem 6.1]{DS98} and we omit the details. 

\begin{lem}\label{lem_john_local}
    There exists a constant $C_0 \geq 1$ which depends only on $M$ such that the following holds true.
    Let $K$ be a $M$-quasiminimal set in $\Omega$ and assume that $N = 2$ or $3$.
    For all connected component $V$ of $\Omega \setminus K$, for all $z \in V$ and for all $0 < r \leq \mathrm{dist}(z,K)$ such that $B(z,C_0 r) \subset \Omega$, there exists a path $\gamma \subset V$ such that
    \begin{enumerate}[label = (\roman*)]
        \item $\gamma \subset V \cap B(z,C_0 r)$,
        \item $\mathrm{dist}(\gamma,K) \geq C_0^{-1} r$,
        \item $\gamma$ goes from $z$ to a point $w \in V \cap B(z,C_0 r)$ such that $\mathrm{dist}(w,K) \geq 2r$.
    \end{enumerate}
\end{lem}

Then it is standard to deduce Theorem~\ref{thm_john} from Lemma \ref{lem_john_local}. We omit the proof again and refer to \cite[Proposition 56.7]{DavidBOOK} or \cite[Lemma 20.1]{BD} for details.

\subsection{A sufficient condition for quasiminimality}\label{section_sufficient}

We finally show that if $K$ is an Ahlfors-regular set which partitions a domain into local John domains with uniform constants, then $K$ satisfies a local quasiminimality property. The statement holds in all dimensions $N \geq 2$.

\begin{thm}\label{thm_char}
    Let $\Omega$ be an open subset of $\R^N$ and let $K$ be a relatively closed subset of $\Omega$. We assume there exists a constant $C_0 \geq 2$ such that the following holds true.
	\begin{enumerate}[label = (\roman*)]
            \item For all $x_0 \in K$ and $r > 0$ such that $B(x_0,r) \subset \Omega$,
		\begin{equation*}    
			    C_0^{-1} r^{N-1} \leq \HH^{N-1}(K \cap B(x_0,r)) \leq C_0 r^{N-1};
		\end{equation*}    
            \item for all $x_0 \in K$ and $r > 0$ such that $B(x_0,r) \subset \Omega$, the ball $B(x_0,r)$ meet at least two components of $\Omega \setminus K$;
            \item for all $x_0 \in K$ and $r > 0$ such that $B(x_0,2r) \subset \Omega$ for all component $V$ of $\Omega \setminus K$, the domain
                \begin{equation*}
                    \bigr(V \cap B(x_0,2r)\bigl) \cup (B(x_0,2r) \setminus \overline{B}(x_0,r))
                \end{equation*}
                is a John domain with constant $C_0$.
	\end{enumerate}
	Then for all $x \in K$, for all $r > 0$ such that $B(x_0,2r) \subset \Omega$ and for all topological competitor $F$ of $K$ in $B(x_0,r)$,
        \begin{equation*}
	    \HH^{N-1}(K\setminus F) \leq M \HH^{N-1}(F\setminus K),
        \end{equation*}
        where $M \geq 1$ is a constant which depends on $N$ and $C_0$.
\end{thm}

The proof of Theorem \ref{thm_char} relies on the following Lemma. 
\begin{lem}\label{lem_char}
    Let $C_0 \geq 1$ be a constant and let $B_0$ be a ball of radius $\geq C_0^{-1}$ such that $2 B_0 \subset B(0,1)$.
    Let $V \subset B(0,1)$ be a John domain containing $B_0$ with an Ahlfors-regular boundary. Then for all open set $W \subset \R^N$ containing $B_0$, we have
    \begin{equation*}
        \HH^{N-1}(\partial V \setminus \overline{W}) \leq C \HH^{N-1}(V \cap \partial W),
    \end{equation*}
    for some constant $C \geq 1$ which depends on $N$, $C_0$, and the John and Ahlfors-regularity constants of $V$.
\end{lem}
\begin{proof}[Proof of Lemma~\ref{lem_char}]
    This is an application of \cite[Lemma 7.12]{DS98}, as done in the proof of \cite[Lemma 7.46]{DS98}.
\end{proof}

\begin{proof}[Proof of Theorem~\ref{thm_char}]
We let $(V_i)_i$ denote the connected components of $\Omega \setminus K$, the sequence being finite or infinite.
We make first a few observations. From property iii), one can see that for all $x_0 \in K$ and $r > 0$ such that $B(x_0,2r) \subset \Omega$, there can only be a finite number of indices $i$ such that $V \cap B(x_0,r/4) \ne \emptyset$ (and this number depends on $N$ and the John constant $C_0$). Thus the family $(V_i)_i$ is locally finite in $\Omega$. Next, let us also observe that for all $x \in K$, there exists at least two indices $i \ne j$ such that $x \in \partial V_i \cap \partial V_j$. This follows from the fact that for all small $r > 0$, $B(x,r)$ meets at least two components of $\Omega \setminus K$ and that the family $(V_i)_i$ is locally finite.

We now fix a point $x_0 \in K$ and a radius $r > 0$ such that $B(x_0,2r) \subset \Omega$, we let $F$ be a topological competitor of $K$ in $B(x_0,r)$. 
    The first step consists in building a family of disjoint open sets $(W_i)_i$ such that
    \begin{enumerate}[label = (\roman*)]
        \item for all $i$, $W_i \setminus B(x_0,r) = V_i \setminus B(x_0,r)$;
        \item for all $i$, $\Omega \cap \partial W_i \subset F$.
    \end{enumerate}
 For this purpose, we set $W_i$ as the union of all components $W$ of $\Omega \setminus F$ such that $W \setminus B(x_0,r) \ne \emptyset$ and $W \setminus B(x_0,r) \subset V_i$. The important point is that for all component $W$ of $\Omega \setminus F$ such that $W \setminus B(x_0,r) \ne \emptyset$, there exists a unique component $V$ of $\Omega \setminus K$ such that $W \setminus B(x_0,r) \subset V$ (by definition of a topological competitor). From there it is easy to check the two properties above.


    Next, we let $m$ denote the number of indices such that such that $V_i \cap B(x_0,r) \ne \emptyset$ (the number $m$ depends on $N$ and $C_0$) and we re-order the sequence $(V_i)_i$ so that $V i \cap B(x_0,r) = \emptyset$ for $i > m$.
    We set for $i = 1,\ldots,m$,
    \begin{align*}
        \hat{V}_i &= \bigl(V_i \cap B(x_0,2r)\bigr) \cup \bigl(B(x_0,2r) \setminus \overline{B}(x_0,r)\bigr)\\
        \hat{W}_i &= \bigl(W_i \cap B(x_0,2r)\bigr) \cup \bigl(B(x_0,2r) \setminus \overline{B}(x_0,r)\bigr)
    \end{align*}
    and we apply Lemma \ref{lem_char} to $\hat{V}_i$ and $\hat{W}_i$ (in particular because they contain any ball of radius $\geq r$ in the ring $B(x_0,2r) \setminus \overline{B}(x,r)$). This yields
    \begin{equation*}
        \HH^{N-1}(B(x_0,r) \cap \partial V_i \setminus \overline{W_i}) \leq C \HH^{N-1}(B(x_0,r) \cap V_i \cap \partial W_i).
    \end{equation*}
    For all $x \in K \setminus F \subset B(x_0,r)$, there exists two different indices $i, j \in \set{1,\ldots,m}$ such that $x \in \partial V_i \cap \partial V_j$. As the family $(W_i)_i$ is disjoint, we have either $x \notin W_i$ or $x \notin W_j$. Since $x \notin F$ and $\Omega \cap \partial W_i \subset F$ (resp. $W_j$), we can deduce that either $x \in \partial V_i \setminus \overline{W_i}$ or $x \in \partial V_j \setminus \overline{W_j}$, whence
    \begin{equation*}
        \HH^{N-1}(K \setminus F) \leq \sum_{i=1}^m \HH^{N-1}(B(x_0,r) \cap \partial V_i \setminus \overline{W_i}).
    \end{equation*}
    On the other hand, we observe that $B(x_0,r) \cap V_i \cap \partial W_i \subset F \setminus K$ whence
    \begin{equation*}
        \sum_{i=1}^m \HH^{N-1}(B(x_0,r) \cap V_i \cap \partial W_i) \leq m \HH^{N-1}(F \setminus K).
    \end{equation*}
\end{proof}

\section{The dimension of junction points}\label{section_dimension}

It is well-known that a minimal set $K$ is regular out a relatively closed subset of dimension $\leq N - 2$ (see for instance \cite[Theorem 4.3]{AFH}). Could there be an analogue property for quasiminimal sets ? A first try is to define regular points as the set of points $x \in K$ such $\lim_{r \to 0} \beta_K(x,r) = 0$. But such a set may not have have a dimension $< N - 1$ in general as we can see on certain Lipschitz graphs. Instead, we shall consider the points $x \in K^*$ as ``regular points'' of quasiminimal sets.
\begin{prop}\label{Jdn}
    Let $K$ be a $M$-quasiminimal set in $\Omega$ and assume that $N = 2$. Then the points of $K \setminus K^*$ are isolated in $\Omega$.
\end{prop}
\begin{proof}
    Assume the contrary. There exists a point $x_0 \in K$ and a radius $r > 0$ that $B(x_0,r) \subset \Omega$ and $B(x_0,r/2)$ contains at least $L$ points of $K \setminus K^*$, where $L$ is a number that we are going to choose soon. Note that for each component $W$ of $B(x_0,r) \setminus K$ which intersects $B(x_0,r/2)$, we have $\HH^{1}(\partial W \cap B(x_0,r/2)) \geq r/2$ (similarly as in Lemma \ref{lem_AF_boundary2D}). As $\HH^{1}(K \cap B(x_0,r)) \leq 2 M \pi r$, we deduce that there can be at most $m \leq 8 \pi M$ such components. 
    We let $W_1, \ldots, W_m$ denote them. Choosing $L= \binom{m}{3}+1$, there exists three boundary, say $\partial W_1$, $\partial W_2$, $\partial W_3$ and two points $x,y \in K \cap B(x_0,r/2)$ such that $x,y \in \partial W_1 \cap \partial W_2 \cap \partial W_3$. For each $i$, there exists a path $\gamma_i$ connecting $x$ to $y$ in $W_i$. Due to Jordan curve theorem one of the paths, say $\gamma_3$, lay in the interior region bounded by two other paths. But since $W_3$ cannot intersects $\gamma_1 \cup \gamma_2$, this implies that $\overline{W_3} \subset B(x_0,r)$. A contradiction to Proposition \ref{prop_isolated}.

\end{proof}


\begin{rmk}
We see from the proof that if $B(x_0,r) \subset \Omega$ and $B(x_0,r) \setminus K$ has $k$ connected components, then the number of points of $K \setminus K^*$ in $B(x_0,r)$ cannot exceed $\binom{k}{3}$.
\end{rmk}

In higher dimension, we can only prove that $K \setminus K^*$ has a dimension $<N-1$ and we will see in Remark \ref{rmk_dimension} that this is optimal.
\begin{prop}
    Let $K$ be a $M$-quasiminimal set in $\Omega$. Then $\mathrm{dim}(K \setminus K^*) \leq N - 1 - \delta$, where $\delta > 0$ is a constant which only depends on $N$ and $M$.
\end{prop}
\begin{proof}
    Corollary \ref{cor_fl} shows that for all $\varepsilon > 0$, there exists a constant $C = C(\varepsilon)$ (which depends on $N$, $M$ and $\varepsilon)$ such that the following holds. For all $x_0 \in K$ and $r > 0$ such that $B(x_0,r) \subset \Omega$, there exists $y \in K \cap B(x_0,r/2)$ and $t \in (C^{-1}r,r/2)$ such that $\beta_K(y,t) \leq \varepsilon$. It follows that for all $z \in K \cap B(y,t/2)$, $\beta_K(z,t/2) \leq 2\varepsilon$. Choosing $\varepsilon$ small enough as in Proposition \ref{prop_plane_separation}, all points $z \in K \cap B(y,t/2)$ are in $K^*$. The fact that $\mathrm{dim}(K \setminus K^*) \leq N - 1 - \delta$, where $\delta$ depends only on $N$ and $M$, is an abstract consequence of this property, as done in \cite[Theorem 51.20]{DavidBOOK} (note that in \cite{DavidBOOK}, $K^*$ plays the role of non regular points whereas it is the inverse for us).
\end{proof}

\begin{rmk}[Example]\label{rmk_dimension}
    Consider a bounded and connected open set $D \subset \R^2$ such that $\partial D$ has a Hausdorff dimension between $1$ and $2$ and each point $x \in \R^2 \setminus D$ admits escape paths to $+\infty$. We identify $\R^2$ to the hyperplane $\set{z = 0}$ in $\R^3$ and we let $K$ be union of $\set{z = 0}$ and the graph of $x \mapsto \mathrm{dist}(x,D)$. 
    The set $K$ is Ahlfors-regular and separates $\R^3$ in three components which admit escape paths to infinity, so $K$ is a quasiminimal set. In this case, we observe that $K \setminus K^* = \partial D$ has a dimension between $N-2$ and $N-1$, where $N = 3$.
\end{rmk}

\vspace{0.5em}

\begin{center}
    { \sc Acknowledgements}
\end{center}

\vspace{0.5em}

The authors are very grateful to Guy David for valuable discussions and insights. During the preparation of this paper, Camille Labourie was funded by the French National Research Agency (ANR) under grant ANR-21-CE40- 0013-01 (project GeMfaceT). Yana Teplitskaya was supported by the Simons Foundation grant 601941, GD.

\bibliographystyle{plain}
\bibliography{biblio_quasiminimalsets}
(C. Labourie) \textsc{Université de Lorraine, CNRS, IECL, F-54000 Nancy, France}\\
(Y. Teplitskaya) \textsc{Université Paris-Saclay, CNRS, Laboratoire de mathématiques d'Orsay, 91405, Orsay, France}
\end{document}